%% file: main.tex
\documentclass[12pt]{amsart}
\usepackage[a4paper, footskip=0.5in,  headheight = 0.5in, top=1.25in, bottom=1.25in,  right=1in,  left=1in]{geometry}
\usepackage{macros}
\allowdisplaybreaks

\title[Induced packing treewidth]{Induced packing treewidth}

\author{Amir Nikabadi$^{\dagger}$}\thanks{$^{\dagger}$IT University of Copenhagen, Denmark (\texttt{amir@itu.dk}).
Supported by the Independent Research Fund Denmark (DFF), grant agreement number 2098-00012B}
\author{Paweł Rzążewski$^{\parallel}$}\thanks{$^{\parallel}$Warsaw University of Technology, Poland (\texttt{pawel.rzazewski@pw.edu.pl}). Supported by the National Science Centre grant 2024/54/E/ST6/00094.}
\date{}

\begin{document}
\begin{titlepage}

\begin{abstract}
In this paper, we introduce a framework that aims to unify classes defined by forbidden induced subgraphs or induced minors with classes defined by the existence of certain structured tree decompositions.
Let $\mathcal{H}$ be a fixed family of graphs. We define \emph{induced-$\mathcal{H}$-packing treewidth},
a tree-decomposition-based graph parameter that, for each bag, measures the maximum number of pairwise
anticomplete induced copies of graphs from $\mathcal{H}$ intersecting that bag. This notion generalizes
some previously studied parameters: when $\mathcal{H}=\{P_1\}$, it is equivalent to tree-independence
number, and when $\mathcal{H}=\{P_2\}$, it is equivalent to induced matching treewidth.

We show that bounded induced-$\mathcal{H}$-packing treewidth yields new algorithmic consequences for a
 range of choices of $\mathcal{H}$.
In particular, we prove the following results for graphs of bounded induced-$\mathcal{H}$-packing treewidth.

\begin{itemize}
\item If $\mathcal{H}$ contains $P_4$, then \textsc{Maximum-Weight Independent Set} (\textsc{MWIS}) can be solved in quasipolynomial time.
\item If $\mathcal{H}$ contains $P_5$, then \textsc{List 3-Coloring} can be solved in quasipolynomial time.

\item  If $\mathcal{H}$ contains $P_3$, then \textsc{Odd Cycle Transversal} and several other problems can be solved in quasipolynomial time.
Furthermore, there is a quasipolynomial-time approximation scheme (QPTAS) for finding a largest induced subgraph of bounded treewidth that satisfies a given hereditary \textsf{CMSO}$_2$ formula.

\item If $\mathcal{H}$ contains all cycles or any path,
then \textsc{MWIS} admits a QPTAS.
\end{itemize}

These results partially answer and substantially extend a question of Bodlaender, Fomin, and Korhonen
[SODA~2026] on the tractability of \textsc{MWIS} for graphs of bounded induced-$\mathcal{H}$-packing
treewidth for $\mathcal{H}=\{P_3\}$ and for $\mathcal{H}$ equal to the family of all cycles.
\end{abstract}
\maketitle

\thispagestyle{empty}
\end{titlepage}

\thispagestyle{empty}
{
  \hypersetup{linkcolor=red!50!black}
  \tableofcontents
}
\thispagestyle{empty}
\newpage\setcounter{page}{1}

\section{Introduction}\label{sec:intro}
\input{sec-intro}

\section{Technical overview}\label{sec:overview}
\input{sec-overview}
\section{Preliminaries}\label{sec:prelim}
\input{sec-prelim}
\section{Container lemmas}\label{sec:containers}
\input{sec-containers}
\section{\MWIS in graphs of bounded induced-$P_4$-packing treewidth}\label{sec:P4}
\input{sec-mwisp4}

\section{\lcol{3} in graphs of bounded induced-$P_5$-packing treewidth}\label{sec:P5}
\input{sec-p5-3coloring}
\section{Algorithms for graphs of bounded induced-$P_3$-packing treewidth}\label{sec:P3}
\input{sec-p3.tex}
\section{Induced packing treewidth and dominated balanced separators}\label{sec:dbs}
\input{sec-dbs}
\section{Computing decompositions}\label{sec:computing}
\input{sec-computing}

\section{Conclusion}\label{sec:conclusion}
\input{sec-conclusion}

\bibliographystyle{abbrvurl}
\bibliography{ref}

\end{document}

%% file: sec-intro.tex
Among classic computational graph problems, \textsc{Max Independent Set} (and its weighted variant) and \textsc{3-Coloring} (and its list variant) are arguably two of the best studied ones.
In the \textsc{Max Weight Independent Set} (\MWIS) problem, we are given a graph $G$ and a weight function $\wei:V(G)\to \mathbb{Q}_{\ge 0}$, and the goal is to find an independent set (i.e., a set of pairwise non-adjacent vertices) of maximum total weight.
In the \lcol{3} problem, we are given a graph $G$ and a list assignment $L:V(G)\to 2^{\{1,2,3\}}$, and the goal is to find a proper coloring of $G$ such that each vertex $v$ is colored with a color from its list $L(v)$.
Both problems are among Karp's 21 \NP-hard~\cite{karp2009reducibility} problems and remain hard for many notions of hardness considered in complexity theory~\cite{cygan2015parameterized,DBLP:journals/eccc/ECCC-TR97-038,DBLP:conf/focs/Khot01,DBLP:journals/tcs/GareyJS76,DBLP:journals/jcss/ImpagliazzoPZ01,DBLP:journals/siamcomp/ChalermsookCKLM20,DBLP:conf/focs/LinRSW23}.
This persistent hardness naturally leads to a central question in algorithmic graph theory:\\
\centerline{\emph{Which restrictions on the input instances make \MWIS and \lcol{3} tractable?}}

\paragraph{Forbidden induced subgraphs.}
Over the years, significant attention has been given to studying classes of instances defined by forbidding certain induced subgraphs.
For graphs $G$ and $H$, we say that $G$ is \emph{$H$-free} if $G$ does not contain an induced subgraph isomorphic to $H$.
For a family of graphs $\cH$, we say that $G$ is \emph{$\cH$-free} if $G$ is $H$-free for every $H \in \cH$.
A special role in the paper is played by classes excluding a fixed path, or, more generally, a fixed linear forest, i.e., a forest of paths. By $P_t$ we denote the path on $t$ vertices. For an integer $k$ and a graph $H$,
by $kH$ we denote the disjoint union of $k$ copies of $H$.

Despite the extensive study of \MWIS in $H$-free graphs, its complexity is still not fully understood.
It is known that the problem is \NP-hard in $H$-free graphs unless $H$ is a forest of subdivided claws~\cite{alekseev1982effect} -- a subcubic forest in which every component has at most one vertex of degree 3 (let us call this family $\cS$). 
On the other hand, polynomial-time algorithms are known only for very restricted cases~\cite{DBLP:journals/jct/Minty80,DBLP:journals/dam/Sbihi80,lozin2008polynomial,BRANDSTADT201857,DBLP:conf/soda/LokshantovVV14,DBLP:journals/talg/GrzesikKPP22}.
However, there is a belief in the community that the problem is actually tractable once we exclude any graph from $\cS$.
This belief is supported by the existence of quasipolynomial-time algorithms.
Indeed, such algorithms were first found for the case that $H$ is a linear forest~\cite{DBLP:conf/focs/GartlandL20,pilipczuk2021quasi}, and later extended to any $H \in \cS$~\cite{DBLP:conf/stoc/GartlandLMPPR24}.

Note that the existence of a quasipolynomial-time algorithm for a given problem is a strong indication that this problem is not \NP-hard, as otherwise every problem in \NP could be solved in quasipolynomial time.

For \lcol{3}, the situation is somewhat similar, though the boundary line between easy and hard cases lies in a different place.
On the one hand, \lcol{3} in $H$-free graphs is \NP-hard unless $H$ is a linear forest~\cite{DBLP:journals/cpc/Emden-WeinertHK98,DBLP:journals/siamcomp/Holyer81a,DBLP:journals/jal/LevenG83}.
On the other hand, polynomial-time algorithms are known only for restricted cases~\cite{DBLP:journals/combinatorica/BonomoCMSSZ18,DBLP:journals/algorithmica/ChudnovskyHSZ21,DBLP:journals/algorithmica/KlimosovaMMNPS20}, but the problem can be solved in quasipolynomial time in $H$-free graphs, for every linear forest $H$~\cite{pilipczuk2021quasi}.

Quite interestingly, the quasipolynomial-time algorithms for the~\MWIS and the \textsc{List 3-Coloring} in graphs excluding a fixed linear forest follow essentially the same approach.
Furthermore, an extension of this approach can be used to solve a certain family of wide generalizations of \MWIS in the graph classes considered here~\cite{DBLP:conf/stoc/GartlandLPPR21}.

For a fixed integer $r$ and fixed \cmsotwo formula\footnote{A \textsf{CMSO}$_2$ formula is a formula in the second-order logic of graphs with quantifiers over vertices and edges. For more details, see, e.g.,~\cite{cygan2015parameterized}.} $\psi$, by $(\tw \leq r,\psi)$-\MWIS we denote the problem where we are given a graph $G$ with vertex-weights, and we want to find a maximum-weight subset $X \subseteq V(G)$ such that the graph induced by $X$ has treewidth at most $r$ and satisfies the formula $\psi$, or report that no such set exists.

This formalism captures a wide range of problems, including \MWIS, \textsc{Feedback Vertex Set} (equivalently, \textsc{Max Induced Forest}), \textsc{Max Induced Matching}, \textsc{Even Cycle Transversal} (equivalently, \textsc{Max Induced Odd Cactus}), and many others.

It turns out that in many cases tractability results for \MWIS in graphs excluding a linear forest can be generalized to $(\tw \leq r,\psi)$-\MWIS, for any $r$ and $\psi$~\cite{DBLP:journals/siamcomp/AbrishamiCPRS24,DBLP:journals/talg/ChudnovskyMPPR26}.
In particular, every problem expressible with this formalism can be solved in quasipolynomial time in $H$-free graphs, for every linear forest $H$~\cite{DBLP:conf/stoc/GartlandLPPR21}.
We remark that such a generalization is not possible if $H$ is not a linear forest, in particular for $H \in \cS$. Indeed, already \textsc{Feedback Vertex Set} is \NP-hard in claw-free graphs, where the \emph{claw} is the three-leaf star $K_{1,3}$~\cite{DBLP:journals/dm/Munaro17}.

\paragraph{Forbidden induced minors.}
What makes classes excluding a linear forest so special?
It turns out that the answer to this question can be given in terms of \emph{induced minors}.
We say that $H$ is an induced minor of $G$ if it can be obtained from $G$ by removing vertices and contracting edges, and $G$ is \emph{$H$-induced-minor-free} if $H$ is not an induced minor of $G$.
Note that if $H$ is a linear forest, then the classes of $H$-free graphs and $H$-induced-minor-free graphs coincide, and this property holds only for linear forests.

Motivated by these observations, Gartland and Lokshtanov (see the doctoral thesis of Gartland~\cite{gartland2023quasi}) posed the following conjecture.

\begin{conjecture}[Gartland and Lokshtanov~\cite{gartland2023quasi}]\label[conjecture]{conj:indminors}
Let $H$ be a planar graph.
Then, for every fixed $r$ and \cmsotwo formula $\psi$, $(\tw \leq r,\psi)$-\MWIS and \lcol{3} can be solved in polynomial time in $H$-induced-minor-free graphs.
\end{conjecture}

We remark that the analogous statement cannot hold for non-planar graphs $H$, as \MWIS and 
\lcol{3} are both \NP-hard in planar graphs, which exclude every non-planar graph as an induced minor.
\cref{conj:indminors} received significant attention in the community, but it appears very challenging.
Obtaining even a quasipolynomial-time algorithm just for \MWIS, or at least a quasipolynomial-time approximation scheme (QPTAS), already seems to be a very ambitious goal. Such results are only known for very restricted cases~\cite{DBLP:conf/stoc/GartlandLMPPR24,DBLP:conf/swat/BonnetCMPR26,chudnovsky2025dominated,DBLP:journals/algorithmica/BonnetDGTW26,DBLP:journals/jctb/BonamyBDEGHTW24,DBLP:journals/jctb/NguyenSS24}. 

\paragraph{Decomposition-based approaches.}
A somewhat different route, complementary to forbidding induced subgraphs or minors, is to consider structural decompositions of graphs. Arguably, the best known example of such a decomposition is \emph{tree decomposition}. Intuitively, tree decomposition aims to embed a graph into a tree by mapping vertices of the graph to \emph{bags} of vertices of the tree, in a way that preserves adjacency and connectivity.
If all bags are small, say of constant size, i.e., the graph has bounded \emph{treewidth}, then the graph ``resembles a tree,'' and many problems can be solved efficiently by mimicking dynamic programming algorithms on trees.
A drawback of such an approach is that its applicability is somewhat limited, as only sparse graphs may have tree decompositions with all bags small. 

As a remedy, one can study \emph{structured tree decompositions}, where we allow bags to be large, but we require that they have some structure that allows us to solve problems efficiently.
For example, bounding the independence number of bags leads to the definition of the \emph{tree-independence number}, introduced independently by Yolov~\cite{yolov2018minor}  and Dallard, Milanič, and Štorgel~\cite{dallard2024treewidth}.
Classes of graphs with small tree-independence number have very strong structural~\cite{HMV25,DBLP:journals/jctb/DallardMS24a,DBLP:conf/esa/LimaMMORS24,chudnovsky2025treewidthcliqueboundednesspolylogarithmictreeindependence} and algorithmic~\cite{yolov2018minor,DBLP:conf/esa/LimaMMORS24,DBLP:journals/corr/abs-2601-15861} properties.
In particular, \lcol{3} (actually, \lcol{$k$} for any fixed $k$) and $(\tw \leq r,\psi)$-\MWIS can be solved in polynomial time in graphs with bounded tree-independence number~\cite{yolov2018minor,DBLP:conf/esa/LimaMMORS24,DBLP:journals/corr/abs-2601-15861}.
A limitation of tree-independence number is that this parameter may be large for some graphs with ``simple structure,'' like complete bipartite graphs (bicliques).

A further generalization of tree-independence number, that in particular captures bicliques, is the notion of \emph{induced matching treewidth}, first introduced by Yolov~\cite{yolov2018minor} under the name \emph{minor-matching hypertreewidth}. Here, for every bag of the decomposition, we aim to bound the size of a maximum induced matching each of whose edges intersects the bag (we emphasize that this matching is not required to be contained in the bag).
Yolov~\cite{yolov2018minor} showed that \MWIS, \lcol{$k$} for any fixed $k$, and some other problems can be solved in polynomial time in graphs of bounded induced matching treewidth.
Later, these results were extended to more problems, including, in particular, \textsc{Feedback Vertex Set},
by Lima, Milanič, Muršič, Okrasa, Rzążewski, and Štorgel~\cite{DBLP:conf/esa/LimaMMORS24}. They also conjectured that for every fixed $r$ and formula $\psi$,  $(\tw \leq r,\psi)$-\MWIS can be solved in polynomial time in graphs of bounded induced matching treewidth.
This conjecture was recently confirmed by Bodlaender, Fomin, and Korhonen~\cite{bodlaender2026finding}.

In the conclusion of their paper, Bodlaender, Fomin, and Korhonen~\cite{bodlaender2026finding} discussed several possible extensions of their result and, in particular, posed the following intriguing question about possible further generalizations of tree-independence number and induced matching treewidth.

\begin{question}[Bodlaender, Fomin, and Korhonen~\cite{bodlaender2026finding}]\label{q:bfk}
    Is \MWIS polynomial-time solvable if the instance is given along with a tree decomposition where every bag is intersected by a bounded number of pairwise disjoint and anticomplete $P_3$s, or cycles?
\end{question}

\subsection{Our contribution}
We introduce a framework that aims to generalize both routes discussed above: forbidding induced subgraphs and using tree decompositions with structure imposed on bags, and in particular to answer \cref{q:bfk}.

\paragraph{Induced $\cH$-packing treewidth.}
Let $\cH$ be a family of graphs. An \emph{induced $\cH$-packing} in a graph $G$ is a collection of pairwise anticomplete induced subgraphs of $G$, each isomorphic to a graph in $\cH$.
For a tree decomposition of $G$, its \emph{induced $\cH$-packing number} is the maximum size of an induced $\cH$-packing whose every element intersects a common bag of the decomposition.
Then, \emph{induced $\cH$-packing treewidth} of $G$, denoted $\treepi_\cH(G)$, is the minimum induced $\cH$-packing number over all tree decompositions of $G$.
If $\cH$ is a singleton, say, $\cH = \{H\}$, we abbreviate the notation and simply write \emph{induced $H$-packing} and $\treepi_H(G)$.

Using the terminology just introduced, tree-independence number is exactly $\treepi_{P_1}$ and induced matching treewidth is exactly $\treepi_{P_2}$.
Furthermore, \cref{q:bfk} concerns graphs of bounded $\treepi_{P_3}$ or $\treepi_{\cC}$, where $\cC$ is the family of all cycles.

On the other hand, the notion of bounded induced $\cH$-packing treewidth can be seen as a ``local relaxation'' of forbidding certain induced substructures.
Consider for example $\cH = \{P_3\}$, as suggested in \cref{q:bfk}.
Graphs $G$ with $\treepi_{P_3} \leq k-1$ may have many pairwise anticomplete $P_3$s,
but no $k$ of them intersect a single bag of some tree decomposition of $G$.
Thus, these graphs are, in some sense, ``locally $kP_3$-free,'' and they generalize $kP_3$-free graphs~\cite{DBLP:conf/esa/GalbyLMN25,DBLP:journals/siamdm/PaesaniPR22}.
Similarly, graphs $G$ with $\treepi_{\cC}(G) \leq k-1$ are a local relaxation of $kC_3$-induced-minor-free graphs~\cite{DBLP:journals/jctb/BonamyBDEGHTW24,DBLP:journals/jctb/NguyenSS24}.

\paragraph{Algorithmic results.}
We address \cref{q:bfk} and more generally, we study the algorithmic consequences of bounded induced packing treewidth for various choices of $\cH$.
While we are not able to obtain polynomial-time algorithms, our framework allows us to solve several problems in quasipolynomial time in graphs of bounded $\treepi_\cH$, for various choices of $\cH$.

For \MWIS we can even go beyond the case of $\treepi_{P_3}$, by showing the following result.

\begin{restatable}{theorem}{thmmwispfour}\label{thm:mwisp4}
Let $\cH$ be a family of graphs containing a path with at most 4 vertices.
For every fixed $k$, given a vertex-weighted graph $(G,\wei)$,
in time $n^{\Oh(\log^2 n)}$ one can either solve \MWIS in $(G,\wei)$,
or correctly report that $\treepi_{\cH}(G)>k$.
\end{restatable}

For \lcol{3}, we can show an analogous result in an even wider class of graphs.

\begin{restatable}{theorem}{thmcolpfive}\label{thm:colpfive}
Let $\cH$ be a family of graphs containing a path with at most 5 vertices.
For every fixed $k$, given an instance $(G,L)$ of \lcol{3},
in time $n^{\Oh(\log^2 n)}$ one can either solve $(G,L)$,
or correctly report that $\treepi_{\cH}(G)>k$.
\end{restatable}

Next, we show a number of algorithmic results in the case that $P_3 \in \cH$.

In the \textsc{Odd Cycle Transversal} (\OCT) problem we are given a vertex-weighted graph $(G,\wei)$ and the goal is to find a minimum-weight set of vertices whose removal makes the remaining graph bipartite.
This problem was only recently shown to be polynomial-time solvable in $kP_3$-free graphs, for any fixed $k$~\cite{DBLP:conf/esa/GalbyLMN25}.
We show the following result concerning ``locally $kP_3$-free'' graphs.

\begin{restatable}{theorem}{thmoctpthree}\label{thm:octpthree}
Let $\cH$ be a family of graphs containing a path with at most 3 vertices.
For every fixed $k$, given a vertex-weighted graph $(G,\wei)$,
in time $n^{\Oh(\log^2 n)}$ one can either solve \OCT in $(G,\wei)$,
or correctly report that $\treepi_{\cH}(G)>k$.
\end{restatable}

Next, using the technique of \emph{blob graphs}~\cite{DBLP:conf/stoc/GartlandLPPR21}, we show that for any fixed $r$ and formula $\psi$, \emph{unweighted} $(\tw \leq r,\psi)$-\MWIS admits a QPTAS in graphs of bounded $\treepi_{P_3}$, provided that the property imposed by $\psi$ is \emph{hereditary}, i.e., closed under vertex deletion and disjoint unions.

\begin{restatable}{theorem}{thmapproxcmsopthree}\label{thm:approx-cmso-p3}
Let $\cH$ be a family of graphs containing a path with at most 3 vertices.
Let $k, r \geq 0$, $\varepsilon \in (0,1)$ be a real, and $\psi$ be a hereditary \textsf{CMSO}$_2$ formula.
Given a graph $G$, in quasipolynomial time one can compute one of the following outputs:
\begin{itemize}
\item a solution to $(\tw \leq r,\psi)$-\MWIS of size at least $(1-\varepsilon)$ times the optimum;
\item a correct conclusion that no solution to $(\tw \leq r,\psi)$-\MWIS exists; or
\item a correct conclusion that $\treepi_{\cH}(G) > k$.
\end{itemize}
\end{restatable}

\paragraph{Dominated balanced separators.}
Now, let us move to a somewhat more general setting in which $\cH$ contains a fixed path or all cycles, as in \cref{q:bfk}. We prove that in both cases, every graph of bounded $\treepi_\cH$ has a small set $X$ of vertices,
such that $N[X]$, i.e., the set consisting of $X$ and its neighbors, is a balanced separator.

\begin{theorem}\label{thm:dbs-together}
    Let $\cH$ be a family of graphs.
    Let $G$ be an $n$-vertex graph with $\treepi_{\cH}(G) \leq k$.
    There is a set $X \subseteq V(G)$ such that every component of $G-N[X]$ has at most $n/2$ vertices, where
    \begin{enumerate}
        \item $|X| = \Oh(kt)$ if $\cH$ contains a path on $t$ vertices, and
        \item $|X| = \Oh(k^3 \log n)$ if $\cH$ contains all cycles.
    \end{enumerate}    
\end{theorem}

Separators dominated by few vertices are believed to exist in $H$-induced-minor-free graphs, for every planar graph $H$ -- this is another conjecture of Gartland and Lokshtanov~\cite{gartland2023quasi}, being a ``structural'' counterpart of \cref{conj:indminors}.
Interestingly, the existence of such separators, combined with some known results~\cite{chudnovsky2025dominated,DBLP:journals/dam/GroenlandORSSS19,chudnovsky2024quasi,DBLP:journals/algorithmica/NovotnaOPRLW21,DBLP:journals/algorithmica/BacsoLMPTL19}, yields several algorithmic consequences (in the $\widetilde{\Oh}(\cdot)$-notation we suppress polylogarithmic factors).

\begin{corollary}
    Let $\cH$ be a family of graphs containing a fixed path or all cycles. For every fixed $k$, the following algorithmic results hold for $n$-vertex graphs $G$ with $\treepi_{\cH}(G) \le k$.
    \begin{enumerate}
        \item \MWIS admits a QPTAS.
        \item \MWIS and \lcol{3} can be solved in time $2^{\widetilde{\Oh}(\sqrt{n})}$.
        \item \textsc{Feedback Vertex Set} (equivalently, \textsc{Maximum Induced Forest}), \textsc{Max Induced Matching}, and \textsc{Max Induced Planar Subgraph} can be solved in time $2^{\widetilde{\Oh}(n^{2/3})}$.
    \end{enumerate}
\end{corollary}

\paragraph{Computing a decomposition.}
Finally, let us discuss the problem of computing a tree decomposition of bounded induced packing treewidth.
It turns out that a simple argument allows us to reduce this problem to the problem of computing a tree decomposition with bounded independence number, which in turn can be done using the algorithm of Dallard, Fomin, Golovach, Korhonen, and Milani\v{c}~\cite{dallard2025computing}.
Summing up, we obtain the following result, which can be seen as a constant-factor \textsf{XP}-approximation algorithm under the natural parameterization.

\begin{restatable}{theorem}{thmcomputedecomposition}\label{thm:computedecomposition}
Let $\cH \neq \emptyset$ be a fixed finite family of connected graphs.
Given an $n$-vertex graph $G$ and an integer $k$,
in time $2^{\Oh(k^2)}n^{\Oh(k)}$ one can either output a tree decomposition of $G$ with induced $\cH$-packing number at most $8k$, or correctly report that $\treepi_{\calH}(G)>k$.
\end{restatable}

One may wonder whether \cref{thm:computedecomposition} can be improved to, say, a fixed-parameter or exact algorithm. This turns out to be unlikely, under standard complexity-theoretic assumptions.
Indeed, another simple reduction allows us to reuse known hardness results for computing tree-independence number~\cite{dallard2024treewidth,dallard2025computing}, yielding the following lower bounds.

\begin{restatable}{theorem}{thmcomputehardness}\label{thm:computehardness}
    Let $\cH \neq \emptyset$ be a fixed finite family of connected graphs.
\begin{enumerate}
        \item It is \NP-hard to decide whether a given graph $G$ satisfies $\treepi_{\cH}(G) \leq 4$.
        \item It is \Wone-hard to approximate induced $\cH$-packing treewidth within any constant factor.
        \item Assuming Gap-ETH, there is no $f(k) \cdot n^{o(k)}$-time $g(k)$-approximation algorithm for  induced $\cH$-packing treewidth, for any computable functions $f$ and $g$.
\end{enumerate}
\end{restatable}

\paragraph*{Organization}
In~\cref{sec:overview} we give a high-level overview of our results and techniques.
In~\cref{sec:prelim} we introduce notation and discuss some basic properties of induced packing treewidth.
In~\cref{sec:containers} we show ``container lemmas'' for \MWIS and \lcol{3} which lie at the heart of our algorithmic results.
We use these lemmas in~\cref{sec:P4} to show \cref{thm:mwisp4} and in~\cref{sec:P5} to show~\cref{thm:colpfive}.
In~\cref{sec:P3} we show~\cref{thm:octpthree} and~\cref{thm:approx-cmso-p3}.
In~\cref{sec:dbs} we prove~\cref{thm:dbs-together}.
In~\cref{sec:computing} we discuss positive and negative results concerning computing tree decompositions of bounded induced packing treewidth; in particular, we prove~\cref{thm:computedecomposition}.
Finally, in~\cref{sec:conclusion} we discuss open problems and directions for future research.

%% file: sec-overview.tex
Let us give a high-level overview of our work.
The main message is that induced packing treewidth provides a common language for two familiar sources of tractability in algorithmic graph theory: excluding induced structures and imposing structure on tree decompositions. This viewpoint yields quasipolynomial-time algorithms for several basic problems as well as structural consequences, and the proofs are organized around a small number of recurring ideas.
In what follows, let $G$ denote the instance graph and let $n$ be the number of its vertices.

\paragraph{Container lemmas.}
The algorithmic core of the paper is a pair of ``container lemmas,'' one for \MWIS and one for \lcol{3},
that allow us to clean the graph with respect to a given bag of a tree decomposition of bounded induced-$\cH$-packing treewidth.
Conceptually, these lemmas isolate the difficult interaction with one bag and replace it with a quasipolynomial-size family of residual instances that are easier to recurse on.
For simplicity, let us assume that $\cH = \{H\}$, where $|V(H)| = h$, but the lemma works for any fixed finite family $\cH$ of graphs.

Both container lemmas follow the same general idea, so let us describe the case of \MWIS, i.e., \cref{thm:H-clean-containers-is}, in more detail.
Consider a tree decomposition of $G$ witnessing that $\treepi_H(G) \leq k$, where $k$ is a constant,
and let $B$ be a bag of this tree decomposition.
We show that one can enumerate a \emph{quasipolynomial-size} family $\cF_{\alpha}$ of pairs $(I,F)$, where $I$ is an independent set and $F \subseteq V(G) \setminus N[I]$, such that:
\begin{enumerate}
\item for each $(I,F) \in \cF_\alpha$, the graph $G[F]$ has no induced copy of $H$ that intersects $B \cap F$, and
\item for every independent set $I^*$ in $G$ there is a pair $(I,F) \in \cF_\alpha$ such that $I \subseteq I^* \subseteq I \cup F$.
\end{enumerate}
The second property implies that, in order to solve \MWIS in $G$,
it is sufficient to solve \MWIS for $G[F]$, for every $(I,F) \in \cF_\alpha$.
Indeed, denoting by $I_F$ the maximum-weight independent set in $G[F]$,
the optimum solution has weight equal to the maximum weight of a set $I \cup I_F$, taken over all $(I,F) \in \cF_\alpha$.

Let us say a few words about the proof of \cref{thm:H-clean-containers-is}.
Essentially, the idea is inspired by the quasipolynomial-time algorithm for \MWIS in $P_t$-free graphs~\cite{pilipczuk2021quasi}: we try to guess $I^*$ by branching on carefully chosen vertices, and stop the search once the first property is satisfied.
The role of bounded induced packing treewidth is to provide a measure of progress: we do not branch blindly, but rather branch so as to eliminate many induced copies of $H$ that still insist on touching the chosen bag $B$.

Consider the family $\cU$ of vertex sets of all largest induced $H$-packings in $G$ in which every copy of $H$ intersects $B$.
Since $B$ is a bag of a tree decomposition witnessing that $\treepi_H(G) \leq k$,
we observe that every member of $\cU$ has at most $kh$ vertices, and thus, $\cU$ has size at most $n^{hk}$.

Now, consider the vertex set $X$ of an induced copy of $H$ that intersects $B$.
Note that $N[X]$, i.e., the set consisting of $X$ and its neighbors, intersects every set $U \in \cU$, as otherwise we could find a larger $H$-packing.
Since $|X| = h$, where $h$ is a constant, there is a vertex $x \in X$ such that $N[x]$ intersects a constant fraction of all elements of $\cU$.
We \emph{branch} on $x$, i.e., create one instance corresponding to putting $x$ into $I^*$ and the other one corresponding to excluding $x$ from $I^*$.
Note that in the first branch we can remove $N[x]$ from the graph, thus shrinking $\cU$ significantly.
Since the size of $\cU$ is polynomial in $n$, we conclude that the recursion tree has $n^{\Oh(\log n)}$ leaves, and we return the corresponding instances as members of $\cF_\alpha$: $I$ denotes the vertices guessed to belong to $I^*$, and $F$ denotes the vertices that were not discarded in the process.

The container lemma for \lcol{3}, i.e., \cref{thm:H-clean-containers-coloring}, is shown similarly, and it enumerates a family $\cF_\chi$ of instances $(G',L')$ of \lcol{3}, where $G'$ is an induced subgraph of $G$, such that
\begin{enumerate}
\item for each $(G',L') \in \cF_\chi$, the graph $G'$ has no induced copy of $H$ that intersects $B \cap V(G')$, and
\item $(G,L)$ is a yes-instance of \lcol{3} if and only if $\cF_\chi$ contains a yes-instance of \lcol{3}.
\end{enumerate}

Summing up, both container lemmas reduce the problem to the case that no copy of $H$ intersects one chosen bag $B$. This yields a common recursive template: clean the graph around $B$, solve recursively the components that lie away from $B$, and use additional structure to handle the components that still touch $B$.
From this point on, the whole story is about understanding what this remaining structure can look like for different choices of $H$ and different optimization problems.

A well-known property of tree decompositions is that there is always a \emph{central bag} $B$ such that every component of $G-B$ has at most $|V(G)|/2$ vertices.
Thus, we apply the appropriate container lemma for the central bag $B$ and consider each component of a member of the returned family $\cF$ (either $\cF_\alpha$ or $\cF_\chi$, depending on the problem) independently.
We distinguish two types of such components: those that do not intersect $B$, and those that do.
By the choice of $B$, the components of the first type are small and can be solved recursively.
This leaves only the components that intersect $B$.

\paragraph{\MWIS and \lcol{3} in graphs with bounded $\treepi_{P_3}$.}
The cleanest setting for this template is bounded induced $P_3$-packing treewidth, as suggested by~\cref{q:bfk}.
In this case, dealing with components that intersect $B$ is easy, as we can show that such components are complete graphs (see \cref{lem:p3clean}).

Thus, bounded $\treepi_{P_3}$ serves as the warm-up case for our framework: once the components touching $B$ collapse to cliques, the recursion closes immediately. For \MWIS, this already resolves (in a weaker form, as we get a quasipolynomial-time, not a polynomial-time algorithm) the $P_3$ side of the motivating question of Bodlaender, Fomin, and Korhonen, and for both \MWIS and \lcol{3} it shows in the cleanest possible form how the framework is meant to work. We can then solve the problem in quasipolynomial time by applying the appropriate container lemma to the central bag $B$ and solving each component of the returned instances either recursively (if it does not intersect $B$) or directly (if it does).
The overall running time of both algorithms is $n^{\Oh(\log^2 n)}$, where one $\log n$ in the exponent comes from the recursion in the container lemma, and the other $\log n$ comes from the recursion on components that do not intersect $B$. A similar running time appears in all our main results, namely \cref{thm:mwisp4,thm:colpfive,thm:octpthree}.

\paragraph{Going beyond $P_3$: \cref{thm:mwisp4,thm:colpfive}.}
The real challenge is to make the same paradigm work once the components of the returned instances that intersect $B$ are no longer trivial. This already happens for longer paths, and this is where the framework has to interact in a serious way with graph structure rather than merely exposing it.

In particular, in the proof of \cref{thm:mwisp4}, i.e., for the case of \MWIS in graphs of bounded induced $P_4$-packing treewidth, we use \emph{modular decomposition}, which allows us to reduce the problem to the case that the graph is \emph{prime} (see~\cref{sec:P4}).
Then, we use the result of Hayward, Hougardy, and Reed~\cite{DBLP:conf/soda/HaywardHR02} that if $G$ is prime, then either it has a simple structure (namely, it is a \emph{split graph}) and thus \MWIS can be solved in polynomial time, or every vertex is contained in an induced $P_4$.
This is exactly the kind of dichotomy that fits the container viewpoint: either the cleaned instance lands in a tractable  class, or induced $P_4$'s remain ubiquitous enough to force further progress. Combining this with the container lemma for \MWIS, we conclude that the problem can be decomposed into a number of instances that are either multiplicatively smaller or have a simple structure.
The overall running time of the algorithm is $n^{\Oh(\log^2 n)}$, as in the case of $H=P_3$.

For \cref{thm:colpfive}, i.e., for the case of \lcol{3} in graphs of bounded induced $P_5$-packing treewidth, we use a different approach, inspired by the work of Hoang, Kamiński, Lozin, Sawada, and Shu~\cite{DBLP:journals/algorithmica/HoangKLSS10} on \lcol{$k$} in $P_5$-free graphs. 
Here the point is different: the residual components need not collapse to a standard graph class, but they still admit enough explicit structure to either reduce to polynomial-time solvable 2-\textsc{Sat} instances or split into multiplicatively smaller subinstances. After applying the appropriate container lemma, we carefully analyze the structure of the components of instances in $\cF_\chi$ that intersect $B$ and show that they can be decomposed in exactly this way.
Again, the overall running time of the algorithm is $n^{\Oh(\log^2 n)}$, as in the case of $H=P_3$.

\paragraph{Back to $P_3$: \textsc{Odd Cycle Transversal} (\cref{thm:octpthree}).}
Returning to the case of $H=P_3$, we next consider the \OCT problem. This is where the framework must cope with a solution consisting of two interacting parts rather than one, so the warm-up case no longer closes for free.
Note that it is equivalent to finding a maximum-weight induced bipartite subgraph of a given graph.
In other words, the goal is to find two disjoint independent sets $X_1, X_2$ of maximum total weight;
we will adapt this viewpoint.

The actual proof is more delicate than in the previous $P_3$-based arguments, because we must control two independent sets simultaneously. In particular, the difficulty is no longer to describe one good residual instance, but to synchronize two such descriptions.
We again apply the container lemma for \MWIS to the graph $G$ and the central bag $B$, and we branch over 
pairs of containers, one for $X_1$ and one for $X_2$.
For every such pair we obtain a residual graph together with information, for each remaining vertex, on whether it may still belong to $X_1$, to $X_2$, or to both.
This yields two layers $V_1$ and $V_2$, where $V_i$ consists of vertices that may still be placed in $X_i$.
In each single layer, each component that intersects $B$ is a clique, as in the previous $P_3$-based arguments.
However, the two layers may interact in a nontrivial way, and the actual components of the residual graph that intersect $B$ may contain multiple components of each layer.

The first key observation is that each vertex in $B$ is simplicial inside each layer, i.e., its neighborhood restricted to one layer is a clique.
This yields a safe reduction rule for a vertex from $B$ that is already forced to one side:
by comparing it with its same-layer neighbors, one can simplify the instance while preserving the optimum value.
The second idea is that edges between vertices that are already forced to different sides can be safely ignored, as they do not interfere with any feasible solution. Thus, we can break the instance into \emph{layer-connected components} -- maximal subgraphs where any two vertices are joined by a path that avoids cross-layer edges.

We show that, after exhaustive application of the reduction rule, each layer-connected component that still intersects $B$ is actually quite simple, since it can be covered by two cliques.
Consequently, such a component has only a few feasible choices for its intersection with $X_1$ and $X_2$ and can be solved directly.
As components that do not intersect $B$ can be solved recursively, we again obtain the running time of $n^{\Oh(\log^2 n)}$. Conceptually, this shows that the framework can also handle coupled solutions, not only single-set objectives such as \MWIS.

\paragraph{Blob graphs in graphs with bounded $\treepi_{P_3}$.}
The case of bounded $\treepi_{P_3}$ has another payoff, this time through the language of \emph{blob graphs}~\cite{gartland2021finding,cameron2006independent}.
For a graph $G$, its blob graph $G^\circ$ has a vertex for each connected set in $G$,
and two sets are adjacent in $G^\circ$ if their union is connected in $G$.
We show that for every graph $G$ it holds that $\treepi_{P_3}(G) = \treepi_{P_3}(G^\circ)$.
This allows us to reduce various problems concerning induced packings in $G$ to \MWIS in appropriate induced subgraphs of $G^\circ$. In other words, bounded $\treepi_{P_3}$ is stable under a transformation that turns packing-type optimization problems back into independent-set problems.
Since $\treepi_{P_3}(G^\circ)=\treepi_{P_3}(G)$ is bounded, we already know that the latter problem can be solved in quasipolynomial time (with respect to the number of vertices of this induced subgraph).
For example, this allows us to find a largest induced $H$-packing, for any fixed $H$,
or a largest set of vertices at pairwise distance at least $d$, where $d$ is even.

Another interesting application of blob graphs is \cref{thm:approx-cmso-p3}: a QPTAS for the unweighted version of $(\tw \leq r, \psi)$-\MWIS for any fixed $r$ and any fixed hereditary \textsf{CMSO}$_2$ formula $\psi$.
The idea comes from~\cite{gartland2021finding}.
Essentially, as a solution has bounded treewidth, it can be broken into constant-size components by removing a small fraction of vertices.
We can then look for a largest induced packing of such constant-size pieces by solving \MWIS in an induced subgraph of the blob graph, at the cost of losing an $\epsilon$-fraction of the optimum value.

\paragraph{Dominated balanced separators.} 
We then turn to a structural consequence, namely \cref{thm:dbs-together} on dominated balanced separators.
Consider an $n$-vertex graph $G$ and its tree decomposition witnessing that $\treepi_{\cH}(G) \leq k$.

Suppose first that $\cH$ contains a path $P_t$.
Starting from the central bag $B$ of the decomposition, we take a maximum induced $P_t$-packing, where every path intersects $B$. Let $Y$ be the vertex set of this packing; we have $|Y| = \Oh(kt)$.
If $N[Y]$ is a balanced separator, then we are done.
Otherwise, there is a unique large component $G'$ of $G - N[Y]$, where large means that it has more than $n/2$ vertices.
Clearly $G'$ intersects $B$, as removing $B$ leaves only small components.
Furthermore, there is no induced $P_t$ in $G'$ intersecting $B$, as otherwise we could add it to the initial packing, contradicting its maximality. Thus the only obstruction to having a balanced separator already in hand is a single large component that still sees the central bag but no longer contains the relevant local pattern through that bag.

We apply the well-known \emph{Gy\'arf\'as path argument}~\cite{gyarfas1987problems} (see also~\cite{DBLP:journals/algorithmica/BacsoLMPTL19}) to $G'$ starting from a vertex $v \in B \cap V(G')$. This yields an induced path $X$ starting at $v$ such that $N[X]$ is a balanced separator, i.e., every component of $G'-N[X]$ is small.
Consequently, $N[X \cup Y]$ is a balanced separator of $G$.
As $X$ has fewer than $t$ vertices, we observe that $X \cup Y$ has at most $\Oh(kt)$ vertices. In this way, the local packing bound around one bag is converted into a global separator by adding only one short induced path.

For the case that $\cH$ contains all cycles, the overall strategy is similar, but already the first step requires a different tool.
Again, we start with the central bag $B$. However, at this point we cannot take $Y$ to be the vertex set of a maximum induced cycle packing, as the cycles in the packing might be of unbounded length.
Instead, we use the recent result of Ahn and Kwon~\cite{Anh-Kwon} concerning the \emph{induced Erd\H{o}s--P\'osa property} for \emph{$B$-cycles}, i.e., cycles that intersect $B$.
They show that if we cannot pack $k+1$ pairwise anticomplete $B$-cycles (which is the case since $B$ is a bag of a tree decomposition witnessing that $\treepi_{\cH}(G) \leq k$), then there is a set $Y$ of $\Oh(k^3)$ vertices such that $N[Y]$ intersects all such cycles. This is exactly what we need from $Y$: in the cycle setting, induced Erd\H{o}s--P\'osa provides a bounded hitting set where the path case used a bounded packing.

Again, if $N[Y]$ is a balanced separator, then we are done, so suppose that there is a unique large component $G'$ of $G-N[Y]$; recall that $G'$ must intersect $B$.
As before, we would like to break $G'$ into small pieces by removing only a few neighborhoods.
However, this time we cannot use the Gy\'arf\'as path argument or any similar tool.
Instead, we use the fact that, after contracting the components of $G'-B$, the resulting graph is a tree,
and thus, it can be broken into small pieces by removing a single vertex.
If this vertex happens to be in $B$, then we are done.
However, if this vertex corresponds to a contracted component $C$ of $G'-B$, we cannot just remove it, as the corresponding component might contain many vertices. Instead, we apply induction to $C$, which is small enough for this purpose since it does not intersect $B$.
Summing up, we obtain a set of size $\Oh(k^3 \log n)$ whose neighborhood is a balanced separator. The extra logarithmic factor is precisely the price of descending through the unique large contracted component.
The $\log n$ factor comes from the induction on the contracted component.

This separator theorem immediately yields a number of algorithmic consequences. As already mentioned in the introduction, the existence of balanced separators dominated by a polylogarithmic number of vertices can be used to obtain a QPTAS for \MWIS~\cite{chudnovsky2024quasi} and subexponential-time algorithms for \lcol{3}~\cite{DBLP:journals/dam/GroenlandORSSS19}, as well as for many problems concerning finding sparse induced subgraphs of large weight~\cite{DBLP:journals/algorithmica/NovotnaOPRLW21,DBLP:journals/algorithmica/BacsoLMPTL19}, such as \textsc{Max Induced Forest} (equivalently, \textsc{Feedback Vertex Set}),  \textsc{Max Induced Matching}, \textsc{Max Induced Planar Subgraph}, or \textsc{Max Induced Subgraph of Degeneracy $d$}, for constant $d$. Notably, as planar and $d$-degenerate graphs may have unbounded treewidth, the last two problems are not captured by the $(\tw \leq r, \psi)$-\MWIS formalism.

\paragraph{Computing decompositions (\cref{thm:computedecomposition,thm:computehardness}).}
Finally, we address the algorithmic prerequisite behind the whole framework: how to compute a tree decomposition of bounded induced $\cH$-packing number. This is essential because all preceding algorithms begin with such a decomposition (or at least its central bag) in hand.

First, we show that for any family $\cH$ of connected graphs and any graph $G$,
$\treepi_{\cH}(G)$ is equal to the tree-independence number of the subgraph of $G^\circ$ induced by all induced copies of graphs from $\cH$ in $G$.
This allows us to reduce the problem of computing a tree decomposition of bounded induced $\cH$-packing number to the already studied problem of computing a decomposition of bounded tree-independence number.
Furthermore, if $\cH$ is a fixed finite family, then the discussed induced subgraph of $G^\circ$ has polynomial size. This yields \cref{thm:computedecomposition}.
This result is important for algorithmic applications, as we always start with a tree decomposition witnessing that $\treepi_{\cH}(G)$ is bounded.

On the other hand, let $G,H$ be graphs, where $H$ is connected,
and let $G * H$ be obtained from $G$ by adding, for each vertex $v \in V(G)$, a copy of $H$ and identifying one vertex of this copy with $v$.
It is easy to see that $\treepi_H(G * H) = \treealpha(G)$.
This simple reduction allows us to lift known hardness results from $\treealpha$ to $\treepi_H$, for any connected $H$. This way we obtain \cref{thm:computehardness}. In particular, \cref{thm:computedecomposition} is close to the best one can hope for at this level of generality. 

%% file: sec-prelim.tex
For a positive integer $n$, by $[n]$ we denote the set $\{1,\ldots,n\}$. For a set $X$, by $2^X$ we denote the family of all subsets of $X$. All logarithms are base 2.

\paragraph{Graphs}
Throughout the paper, graphs have finite vertex sets, no loops, and no parallel edges. 
Let $G$ be a~graph with vertex set $V(G)$ and edge set $E(G)$.
For $X \subseteq V(G)$, we denote the subgraph of $G$ \emph{induced by $X$} as $G[X]$,
that is, $G[X] = (X, \{uv ~|~ u, v \in X \mbox{ and } uv \in E(G) \})$,
and we write $G-X$ for $G[V(G) \setminus X]$.
When it does not lead to confusion, we will sometimes identify induced subgraphs with their vertex sets.

A graph $G$ contains a graph $H$ as a \emph{contraction} if we can obtain $H$ from $G$ by a series of edge contractions.
Note that, unlike induced minors, we are not allowed to delete vertices.

By $N(v)$ we denote the set of neighbors of a vertex $v$ in $G$, and we use $N[v]$ for $N(v) \cup \{v\}$.
For a set $X \subseteq V(G)$, we define $N[X] = \bigcup_{v \in X} N[v]$ and $N(X) = N[X] \setminus X$.
Two disjoint sets $X,Y \subseteq V(G)$ are \emph{complete} (resp., \emph{anticomplete}) if all edges between them exist (resp., there are no edges between them). 
By $\alpha(G)$ we denote the size of a largest independent set in $G$.

A \emph{vertex-weighted graph} is a pair $(G,\wei)$ where $G$ is a graph and $\wei:V(G)\to\mathbb Q_{\ge0}$ is a weight function. For a set $X\subseteq V(G)$, we write $\wei(X)=\sum_{v\in X}\wei(v)$, and we write $\wei(G):=\wei(V(G))$.
We assume that all computations on weights can be performed in constant time.

\paragraph{Tree decompositions and balanced separators}
For a graph $G$, a \emph{tree decomposition} of $G$ is a pair $(T,\beta)$, where $T$ is a tree and $\beta \colon V(T) \rightarrow 2^{V(G)}$ is a map with the following properties:
\begin{itemize}
    \item For every $v \in V(G)$, there exists $t \in V(T)$ such that $v \in \beta(t)$.
    \item For every $uv \in E(G)$, there exists $t \in V(T)$ such that $u, v \in \beta(t)$.
    \item For every $v \in V(G)$, the subgraph of $T$ induced by $\{ t \in V(T) ~|~ v \in \beta(t) \}$ is connected.
\end{itemize}
For every $t \in V(T)$, we refer to $t$ as a \emph{node} of $T$, and to $\beta(t)$ as a \emph{bag} of $(T, \beta)$.
The \emph{width} of a tree decomposition $(T, \beta)$ is $\max_{t\in V(T)} |\beta(t)|-1$.
The \emph{treewidth} of $G$, denoted by $\tw(G)$, is the minimum width of a tree decomposition of $G$.
The \emph{independence number} of a tree decomposition $(T, \beta)$ of $G$ is $\max_{t\in V(T)} \alpha(G[\beta(t)])$. The \emph{tree-independence number} of $G$, denoted $\treealpha(G)$, is the minimum independence number of a tree decomposition of $G$.

For a weighted graph $(G,\wei)$, a \emph{balanced separator} of $G$ is a set $S \subseteq V(G)$ such that every component of $G-S$ has weight at most $\wei(G)/2$. We will often work with balanced separators for \emph{uniform} weight functions, i.e., $\wei(v)=1$ for every $v\in V(G)$. In this case, every component of $G-S$ has at most $|V(G)|/2$ vertices.

The following lemma is a standard fact about tree decompositions and balanced separators (e.g.~\cite[Lemma~7.19]{cygan2015parameterized}).

\begin{lemma}\label{lem:balanced-bag}
Let $(G,\wei)$ be a weighted graph with a tree decomposition $(T,\beta)$.
Then there is a node $t\in V(T)$ such that
$\beta(t)$ is a balanced separator of $(G,\wei)$.
\end{lemma}

\paragraph{$B$-rooted subgraphs and $(B,\cH)$-clean sets}
Let $G$ be a graph and $B \subseteq V(G)$.
Any induced subgraph of $G$ whose vertex set intersects $B$ is called \emph{$B$-rooted}.

We say a set $C\subseteq V(G)$ is \emph{$(B,\cH)$-clean} if
$G[C]$ has no $B$-rooted induced subgraph isomorphic to any member of $\cH$.
If $V(G)$ is $(B,\cH)$-clean, we say that $G$ is $(B,\cH)$-clean.

\paragraph{Induced packings and induced packing treewidth.}
 Let $\cH$ be a family of graphs.
An \emph{induced $\cH$-packing} in $G$ is a set of pairwise anticomplete induced subgraphs of $G$, each isomorphic to a member of $\cH$.
An induced $\cH$-packing is \emph{$B$-rooted} if every member of the packing intersects $B$.
The size of the largest $B$-rooted induced $\cH$-packing in $G$ is denoted by $\pi_{\cH}(G,B)$.

For a tree decomposition $\cT=(T,\beta)$ of $G$, we define $\pi_{\cH}(G,\cT) = \max_{t \in V(T)} \pi_{\cH}(G, \beta(t))$.
Then, $\treepi_{\cH}(G) = \min_{\cT} \pi_{\cH}(G,\cT)$ where the minimum is taken over all tree decompositions $\cT$ of $G$.
When $\cH$ is a singleton $\{H\}$, we abbreviate the notation above by writing $H$ instead of $\{H\}$.

Let us record three basic properties of induced packing treewidth that we will use later.

\begin{lemma}\label{lem:structural-package}
Let $\cH,\cH'$ be families of graphs and $G,G'$ be graphs. The following hold.
\begin{enumerate}
\item\label{sprop:3} If $G'$ is an induced subgraph of $G$, then $\treepi_{\cH}(G')\leq \treepi_{\cH}(G)$.

\item\label{sprop:4} If $\cH'\subseteq \cH$, then $\treepi_{\cH'}(G)\leq \treepi_{\cH}(G)$.

\item\label{sprop:7} For every $t\geq 1$, one has $\treepi_{P_{t+1}}(G)\leq \treepi_{P_t}(G)$.

\item\label{sprop:indminors} Let $H$ be a graph and let $\cH$ be the family of all graphs that contain $H$ as a contraction.
If $G'$ is an induced minor of $G$, then $\treepi_{\cH}(G') \leq \treepi_{\cH}(G)$.
\end{enumerate}
\end{lemma}

\begin{proof}
For~(\ref{sprop:3}), let $(T,\beta)$ be a tree decomposition of $G$.
For each $t\in V(T)$, define $\beta'(t)=\beta(t)\cap V(G')$.
Then $(T,\beta')$ is a tree decomposition of $G'$.
Moreover, every $\beta'(t)$-rooted induced $\cH$-packing in $G'$ is also a $\beta(t)$-rooted induced $\cH$-packing in $G$, so
\[
\pi_{\cH}(G',\beta'(t))\leq \pi_{\cH}(G,\beta(t))
\]
for every $t\in V(T)$.
Hence $\pi_{\cH}(G',(T,\beta'))\leq \pi_{\cH}(G,(T,\beta))$, and taking the minimum over tree decompositions of $G$ yields $\treepi_{\cH}(G')\leq \treepi_{\cH}(G)$.

For~(\ref{sprop:4}), every induced $\cH'$-packing is an induced $\cH$-packing.
Hence $\pi_{\cH'}(G,B)\leq \pi_{\cH}(G,B)$ for every $B\subseteq V(G)$, and therefore $\treepi_{\cH'}(G)\leq \treepi_{\cH}(G)$.

For~(\ref{sprop:7}), fix any set $B \subseteq V(G)$ and let $r=\pi_{P_{t+1}}(G,B)$.
Then there are $r$ pairwise anticomplete induced paths $P^1,\ldots,P^r$ on $t+1$ vertices, each intersecting $B$.
For each $i\in [r]$, the first $t$ vertices and the last $t$ vertices of $P^i$ form induced $t$-vertex subpaths whose union is $V(P^i)$, so at least one of them intersects $B$; choose one such subpath and call it $\widehat{P}^i$.
Then $\widehat{P}^1,\ldots,\widehat{P}^r$ form a $B$-rooted induced $P_t$-packing in $G$.
Therefore $\pi_{P_t}(G,B)\geq \pi_{P_{t+1}}(G,B)$.
Since this holds for every $B \subseteq V(G)$, we conclude that $\treepi_{P_{t+1}}(G)\leq \treepi_{P_t}(G)$.

Finally, let us show~(\ref{sprop:indminors}).
By (\ref{sprop:3}), $\treepi_{\cH}$ does not increase when we delete vertices,
it is sufficient to argue that it also does not increase after a single edge contraction.
So assume that $G'$ was obtained from $G$ by contracting an edge $xy$ into a new vertex $z$.
Let $\cT=(T,\beta)$ be a tree decomposition of $G$ with minimum $\pi_H(G,\cT)$.
Let $\cT' = (T,\beta')$ be obtained from $\cT$ by replacing, in every bag, every occurrence of $x$ or $y$ by $z$.
It is standard to verify that $\cT'$ is a tree decomposition of $G'$.

Consider a node $t \in V(T)$ and a maximum $\beta'(t)$-rooted induced $\cH$-packing in $G'$.
If no member of this packing contains $z$, then the same vertices form an induced $\cH$-packing in $G$ (of the same size).
Otherwise, replacing $z$ by $\{x,y\}$ gives an induced $\cH$-packing in $G$ of the same size; here we use the assumption on $\cH$.
This $\cH$-packing is $\beta(t)$-rooted, as $z \in \beta'(t)$ implies that $\{x,y\} \cap \beta(t) \neq \emptyset$.
Thus, for any $t \in V(T)$, it holds that $\pi_{\cH}(G',\beta'(t)) \leq \pi_H(G,\beta(t)) \leq \pi_\cH(G,\cT) = \treepi_{\cH}(G)$.
Consequently, $\treepi_{\cH}(G') \leq \treepi_{\cH}(G)$.
\end{proof}

A tempting generalization of~\Cref{lem:structural-package}~(\ref{sprop:7}) would be to say that if $H'$ is an induced subgraph of $H$, then $\treepi_{H}(G)\leq \treepi_{H'}(G)$ for every graph $G$.
This is, however, false.
Let $K_3$ be the triangle -- the complete graph on 3 vertices, and $\mathsf{Paw}$ be the graph obtained from $K_3$ by adding one vertex adjacent to exactly one vertex of the triangle.
For every $k$, consider the graph $G_k$ defined as follows. 
We start by introducing two disjoint sets of vertices $A=\{a_1,\ldots,a_k\}$ and $B=\{b_1,\ldots,b_k\}$, 
and adding all edges between $A$ and $B$.
Then, for every vertex $v\in A\cup B$, we introduce three new vertices
$p_v,q_v,r_v$ and add the edges $vp_v,p_vq_v,q_vr_v$, and $r_vp_v$.
Note that vertices $v,p_v,q_v,r_v$ induce a copy of $\mathsf{Paw}$.

First, we observe that $\treepi_{K_3}(G_k)= 1$.
Indeed, consider the tree decomposition with one central bag $A\cup B$, and, for every $v\in A\cup B$, one leaf bag $\{v,p_v,q_v,r_v\}$ adjacent to the central bag.
The central bag intersects no triangle, and each leaf bag intersects only the ``private'' triangle $p_vq_vr_v$. Hence every bag intersects at most one triangle.

We now show that $\treepi_{\mathsf{Paw}}(G_k)\geq k$.
Consider an arbitrary tree decomposition of $G_k$.
Restricting attention to the complete bipartite graph $G_k[A\cup B]$, the standard subtree argument shows that some bag contains all vertices of $A$ or all vertices of $B$.
By symmetry, assume the former.
Then this bag intersects $k$ induced and pairwise anticomplete paws, one for each vertex $v \in A$.
This shows that $\treepi_{\mathsf{Paw}}(G_k)\geq k$.

The same construction can be used to show that the assumption on $\cH$ in \Cref{lem:structural-package}~(\ref{sprop:indminors}) cannot be omitted.
Indeed, let $G'_k$ be the graph obtained from $G_k$ by contracting all edges of the form $vp_v$, for every $v\in A\cup B$.
Clearly, $G'_k$ is an induced minor of $G_k$.
For each original vertex $v\in A\cup B$, the contracted image of $\{v,p_v,q_v,r_v\}$ induces a triangle in $G'_k$.
Therefore, the same subtree argument applied to the contracted bipartite core of $G'_k$ shows that some bag intersects $k$ pairwise anticomplete triangles.
Hence $\treepi_{K_3}(G'_k)\geq k$.

%% file: sec-containers.tex
In this section we show two algorithms that, given a graph $G$ and a set $B \subseteq V(G)$ with $\pi_{\cH}(G,B) \leq k$, compute a family of induced subgraphs of $G$ that are $(B,\cH)$-clean and preserve a solution.
In the first lemma, the solution is a (maximum-weight) independent set, while in the second lemma it is a (list) 3-coloring.

\begin{lemma}[Container lemma for independent sets]\label{thm:H-clean-containers-is}
Let $\cH \neq \emptyset$ be a fixed finite family of graphs.
For every fixed integer $k$, there is an algorithm which,
given an $n$-vertex graph $G$ and a set $B\subseteq V(G)$ satisfying $\pi_{\cH}(G,B)\leq k$,
in time $n^{\Oh(\log n)}$ computes a family $\cF_\cH(G,B)$
of size $n^{\Oh(\log n)}$ such that the following hold:
\begin{enumerate}[\rm (i)]
    \item each member of $\cF_\cH(G,B)$ is a pair $(I,F)$, where $I$ is an independent set and $F \subseteq V(G) \setminus N[I]$,
    \item for every $(I,F) \in \cF_\cH(G,B)$, the subgraph $G[F]$ is
    $(B \cap F,\cH)$-clean, that is, $G[F]$ has no induced copy of a member of
    $\cH$ intersecting $B \cap F$, and
    \item for every independent set $I^*$ in $G$ there is $(I,F) \in \cF_\cH(G,B)$
    such that $I \subseteq I^* \subseteq I \cup F$.
\end{enumerate}
\end{lemma}

\begin{proof}
We proceed by induction on $k$.
If $k=0$, then $\pi_{\cH}(G,B) = 0$ implies that $G$ is $(B,\cH)$-clean.
In this case, the one-element family containing the pair $(\emptyset,V(G))$ has the desired properties.
Thus we may assume that $k\geq 1$ and the lemma holds for all smaller values of $k$.

The following claim encapsulates the main recursive step of the algorithm.

\begin{claim}\label{clm:contained-recursion}
In time $n^{\Oh(\log n)}$, one can compute a family $\cF_k$ of $n^{\Oh(\log n)}$ pairs $(I,F)$,
where $I$ is an independent set and $F \subseteq V(G) \setminus N[I]$,
such that:
\begin{itemize}
    \item for every $(I,F) \in \cF_k$, it holds that $\pi_{\cH}(G[F],B \cap F) \leq k-1$, and
    \item for every independent set $I^*$ of $G$ there is $(I,F) \in \cF_k$ such that $I \subseteq I^* \subseteq I \cup F$.
\end{itemize}
\end{claim}
\begin{claimproof}
Let $h$ be the number of vertices in a largest member of $\cH$.
We can imagine the procedure as a recursive process building a tree $\mathsf{T}$ whose nodes are labeled by pairs $(I,F)$, where $I$ is an independent set and $F \subseteq V(G) \setminus N[I]$.

The process is essentially a branching algorithm that guesses some vertices to be included in the independent set $I^*$. The elements of $I$ are selected to be included in $I^*$, while the elements of $F$ are candidates for inclusion in $I^*$.

We start with a tree containing a single node, labeled by $(\emptyset,V(G))$.
Now, we proceed as follows.
Consider a node of $\mathsf{T}$ that has not yet been processed, let it be labeled by $(I,F)$.
Define $B' = B \cap F$.
If $\pi_{\cH}(G[F],B') \leq k-1$, then we do not process this node any further, it will be a leaf of the tree.
Thus, suppose that $\pi_{\cH}(G[F],B') = k$.

Let $\cU(F)$ be the family of all $B'$-rooted induced $\cH$-packings in $G[F]$ of size $k$.
Each member of $\cU(F)$ uses at most $kh$ vertices in total, so $|\cU(F)| \leq n^{kh}$.
In particular, $|\cU(F)|$ is polynomial in $n$ because both $k$ and $h$ are constant.
On the other hand, $\cU(F) \neq \emptyset$ by the assumption on $G[F]$.

Let $U$ be a member of $\cU(F)$, and let $X$ be the vertex set of one member of this packing, i.e., $G[X]$ is an induced copy of a graph in $\cH$ intersecting $B'$.
Note that $N[X]$ intersects the vertex set of every packing in $\cU(F)$.
Indeed, otherwise we could add $X$ to a packing in $\cU(F)$ that is disjoint from $N[X]$, obtaining a $B'$-rooted induced $\cH$-packing of size $k+1$.
As $|X| \leq h$, there is a vertex $x \in X$ such that $N[x]$ intersects at least $|\cU(F)|/h$ members of $\cU(F)$. 

We create two nodes in the recursion tree as children of the currently processed node:
one will be labeled with $(I \cup \{x\}, F-N[x])$ and the other with $(I, F-\{x\})$.
Note that this corresponds to a natural branching whether we include $x$ in the constructed independent set or not.

We continue this process exhaustively until all leaves of $\mathsf{T}$ are labeled with pairs $(I,F)$ where $\pi_{\cH}(G[F],B \cap F) \leq k-1$. We include these pairs in the family $\cF_k$.

The second property of $\cF_k$ follows from the exhaustive branching process: for every independent set $I^*$ compatible with the current node, if $x \in I^*$ we move to the child $(I \cup \{x\},F-N[x])$, and otherwise we move to the child $(I,F-\{x\})$.

Now let us argue about the size of $\cF_k$.
Note that if a node is labeled with $(I,F)$ and its child node is labeled with $(I',F')$, then $F' \subseteq F$ and hence, $\cU(F') \subseteq \cU(F)$.
Thus, we can use the size of $\cU(F)$ as a potential of the current node.
Consequently, the number of leaves in the recursion tree is bounded by the following inequality
\[
    F(\mu) \leq F(\mu - 1) + F(\lfloor (1-1/h)\mu \rfloor).
\]
This is solved by $F(\mu) = \mu^{\Oh(\log \mu)}$. As $\mu \leq n^{kh}$, we obtain that the number of leaves in $\mathsf{T}$, and thus the size of $\cF_k$, is bounded by $n^{\Oh(\log n)}$.

Finally, note that all local operations, including testing whether $\pi_{\cH}(G[F],B\cap F) \leq k-1$ and computing the family $\cU(F)$, can be performed in time $n^{\Oh(kh)}$, which is polynomial because $k$ and $h$ are constant.
Consequently, the total running time of the algorithm is $n^{\Oh(\log n)}$, as claimed.
\end{claimproof}

Now, apply \cref{clm:contained-recursion} to compute the family $\cF_k$.
For every $(I,F) \in \cF_k$, recursively compute the family $\cF_{(I,F)} := \cF_{\cH}(G[F],B \cap F)$ of pairs $(I',F')$, where
$I' \subseteq F$ is an independent set and $F' \subseteq F \setminus N[I']$ is a $(B\cap F,\cH)$-clean container in $G[F]$.
Finally, we set
\[
\cF_\cH(G,B)=\bigcup_{(I,F) \in \cF_k} \bigcup_{(I',F') \in \cF_{(I,F)}}  (I \cup I',F').
\]

For every resulting pair $(I \cup I',F')$, the set $I \cup I'$ is independent and $F' \subseteq V(G) \setminus N[I \cup I']$.
The running time and the size of the computed family are upper-bounded by $n^{\Oh(\log n)}$ (where the dependence on $k$ and $\cH$ is hidden in the $\Oh(\cdot)$-notation).
The last property of $\cF_\cH(G,B)$ follows from the second property of $\cF_k$ and the induction hypothesis.
This completes the proof.
\end{proof}

\begin{lemma}[Container lemma for list 3-colorings]\label{thm:H-clean-containers-coloring}
Let $\cH \neq \emptyset$ be a fixed finite family of graphs.
For every fixed integer $k$, there is an algorithm which,
given an $n$-vertex graph $G$ with list function $L : V(G) \to 2^{\{1,2,3\}}$ and a set $B\subseteq V(G)$ satisfying $\pi_{\cH}(G,B)\leq k$,
in time $n^{\Oh(\log n)}$ either detects that $(G,L)$ is a no-instance or computes a family $\cF_\cH(G,B)$
of size $n^{\Oh(\log n)}$ such that the following hold:
\begin{enumerate}[\rm (i)]
    \item each member of $\cF_\cH(G,B)$ is a pair $(G',L')$, where $G'$ is an induced subgraph of $G$ and $L' : V(G') \to 2^{\{1,2,3\}}$ is a list function,
    \item for every $(G',L') \in \cF_\cH(G,B)$, the subgraph $G'$ is
    $(B \cap V(G'),\cH)$-clean,
    \item $(G,L)$ admits a list 3-coloring, if and only if there is $(G',L') \in \cF_\cH(G,B)$ that admits a list 3-coloring.
\end{enumerate}
\end{lemma}

Essentially, the idea of the proof of \cref{thm:H-clean-containers-coloring} is to modify the proof of \cref{thm:H-clean-containers-is} in the same way as a quasipolynomial-time algorithm for \lcol{3} in $P_t$-free graphs is obtained by modifying a quasipolynomial-time algorithm for \MWIS in the same class of graphs~\cite{pilipczuk2021quasi}.
Thus, we will only sketch it.

\begin{proof}[Sketch of proof.]
    Let $h$ be the number of vertices in a largest member of $\cH$.
    First of all, if there is a vertex with empty list, we reject the instance as it has no solution.
    Second, if there is a vertex $v$ with a singleton list, say $L(v) = \{i\}$, we remove $v$ from the graph and remove the color $i$  from the lists of its neighbors, obtaining an equivalent instance.
    
    Third, we exhaustively enumerate all induced subgraphs of $G$
    with at most $kh$ vertices, together with their list 3-colorings respecting $L$.
    If some induced subgraph of $G$ has no coloring, we reject the instance as it has no solution.
    Furthermore, if there is $S \subseteq V(G)$ of size at most $kh$, a vertex $v \in S$, and $i \in L(v)$,
    such that $G[S]$ has no list 3-coloring using lists $L$ with $v$ colored with $i$, we remove color $i$ from the list of $v$.

    We apply the above three steps exhaustively until no further changes are possible.
    This can be done in polynomial time as $k$ and $h$ are constant, and every step reduces the total size of lists.

    Then, we proceed similarly as in the proof of \cref{thm:H-clean-containers-is}.
    The crucial step is the recursive construction of a tree $\mathsf{T}$, analogous to the one in 
    \cref{clm:contained-recursion}.
    This time, nodes of $\mathsf{T}$ are labeled with pairs $(G',L')$, where $G'$ is an induced subgraph of $G$ and $L' : V(G') \to 2^{\{1,2,3\}}$ is a list function.
    Whenever a child node is created, we first apply the three reduction steps above exhaustively to its instance before processing that node further.

    Again, suppose that we are processing a node labeled with $(G',L')$, where, for $B' = B \cap V(G')$, we have $\pi_{\cH}(G',B') = k$.
    Let $\cU(G')$ be the family of all $B'$-rooted induced $\cH$-packings in $G'$ of size $k$.
    Let $\cU^*(G')$ be the family of colored elements of $\cU(G')$, i.e., pairs $(U,\phi)$, where $U \in \cU(G')$ and $\phi$ is a list 3-coloring of $G'[U]$ using lists $L'$.
    Note that
    $
|\cU^*(G')| \leq |\cU(G')|  \cdot 3^{kh} \leq n^{kh} \cdot 3^{kh} =  n^{\Oh(1)}
    $
    and it can be computed in polynomial time as $k$ and $h$ are constant.
    Furthermore, because we exhaustively apply the reduction rules at every node,
    we know that for every $U \in \cU(G')$, every $v \in U$, and every color $i \in L'(v)$,
    there is a member $(U,\phi) \in \cU^*(G')$ such that $\phi(v) = i$.

    Recall that in the proof of \cref{clm:contained-recursion}, we showed that if we choose any packing $U \in \cU(G')$ and then any member of $U$, there is a vertex $x$ of that member such that $N[x]$ intersects at least $|\cU(G')|/h$ members of $\cU(G')$.
    Let $\mathcal{P}$ be the set of packings in $\cU(G')$ whose vertex set intersects $N[x]$; then $|\mathcal{P}| \geq |\cU(G')|/h$.
    For each packing $U \in \mathcal{P}$, choose one vertex $y_U \in V(U) \cap N[x]$.
    Since every current list has size at least 2, each $L'(y_U)$ is one of the four sets $\{1,2\}, \{1,3\}, \{2,3\}, \{1,2,3\}$.
    Hence, for some set $A \subseteq \{1,2,3\}$ with $|A| \geq 2$, there is a subfamily $\mathcal{P}' \subseteq \mathcal{P}$ of size at least $|\mathcal{P}|/4$ such that $L'(y_U)=A$ for every $U \in \mathcal{P}'$.
    Since both $A$ and $L'(x)$ have size at least 2, we may choose a color $i \in A \cap L'(x)$.

    We branch on whether we color $x$ with color $i$ or not.
    If we decide not to color $x$ with color $i$, we remove $i$ from the list of $x$; this is the analog of not including $x$ in the independent set in the proof of \cref{clm:contained-recursion}.
    If we decide to color $x$ with color $i$, we remove $x$ from the graph and remove color $i$ from the lists of its neighbors; this is the analog of including $x$ in the independent set in the proof of \cref{clm:contained-recursion}. 

    We measure the progress of the branching by the size of $\cU^*(G')$.    
    In the first case, if we remove color $i$ from the list of $x$, the size of $\cU^*(G')$ decreases by at least 1 (as $x$ belongs to a member of $\cU(G')$ and there is a coloring in $\cU^*(G')$ that uses color $i$ for $x$).
    
    In the second case, if we color $x$ with color $i$, then for every $U \in \mathcal{P}'$ the vertex $y_U$ either disappears (if $y_U=x$) or loses the color $i$ from its list.
    By the reduction invariant above, for every $U \in \mathcal{P}'$ there exists a coloring $\phi_U$ of $G'[U]$ such that $\phi_U(y_U)=i$.
    Hence each pair $(U,\phi_U)$ is destroyed in this branch, so the size of $\cU^*(G')$ decreases by at least $|\mathcal{P}'|$.
    Therefore, the size of $\cU^*(G')$ decreases by at least
    \[
        |\mathcal{P}'| \geq |\mathcal{P}|/4 \geq |\cU(G')|/(4h) \geq |\cU^*(G')|/(4h \cdot 3^{kh}),
    \]
    that is, by a constant fraction.

    As $|\cU^*(G')| = n^{\Oh(1)}$, the number of leaves in the recursion tree is bounded by $n^{\Oh(\log n)}$.
    The rest of the proof is analogous to the proof of \cref{thm:H-clean-containers-is}. 
\end{proof}

%% file: sec-mwisp4.tex
In this section we prove \cref{thm:mwisp4}.
\thmmwispfour*

\subsection{Dealing with $(B,P_4)$-clean instances}

Let us start with $(B,P_4)$-clean graphs, i.e., graphs that do not contain an induced 4-vertex path intersecting a given set $B$ of vertices. Note that such graphs arise as the output of \cref{thm:H-clean-containers-is} applied to $\cH = \{P_4\}$.

\paragraph{Modules and modular decomposition}
Let $G$ be a graph.
A nonempty set $M\subseteq V(G)$ is a \emph{module} of $G$ if every vertex of
$V(G)\setminus M$ is either complete to $M$ or anticomplete to $M$.
A module is \emph{trivial} if it is a singleton, or all of $V(G)$,
otherwise it is \emph{nontrivial}.
A graph is \emph{prime} if it has no nontrivial modules.

Let $\cM = \{M_1,\ldots,M_r\}$ be a partition of $V(G)$ into nonempty modules.
The \emph{quotient graph} of $G$ with respect to $\cM$, denoted $G/\cM$, is the graph with vertex set $\{q_1,\ldots,q_r\}$, where $q_i$ corresponds to the module $M_i$, and where $q_iq_j\in E(G/\cM)$ if and only if $M_i$ is complete to $M_j$ in $G$.

The theory of modules is quite rich, and we refer the reader to the survey of Habib and Paul~\cite{habib2010survey} for a comprehensive overview. We will only need a few basic facts about modules and modular decomposition, which we summarize in the following lemma.

\begin{lemma}[Modular decomposition tree]\label{lem:modulardecomposition}
    Given a graph $G$, in linear time one can compute a rooted tree $\mathsf{T}$, called the \emph{modular decomposition tree} of $G$, whose nodes correspond to nonempty modules of $G$, such that:
    \begin{enumerate}
        \item the root of $\mathsf{T}$ corresponds to the module $V(G)$, and the leaves of $\mathsf{T}$ correspond to the singleton modules of $G$,
        \item for every non-leaf node $M$ of $\mathsf{T}$, the labels of its children form a partition of the vertex set of $G[M]$ into nonempty modules, and
        \item for every non-leaf node $M$ of $\mathsf{T}$, the quotient graph corresponding to the partition of $G[M]$ given by its children is either complete, edgeless, or prime.
    \end{enumerate}
\end{lemma}

Let us emphasize that for a graph $G$ and a partition of $V(G)$ into nonempty modules, the quotient graph is an induced subgraph of $G$.
Indeed, picking one vertex from each module gives an induced subgraph of $G$ isomorphic to the quotient graph.
Thus, if $G$ is, say, $H$-free for some graph $H$, then so is the quotient graph.

\paragraph{Using the modular decomposition tree to solve \MWIS}
It is well known that modular decomposition can be used to reduce solving various problems to the case of prime graphs~\cite{habib2010survey}. Let us discuss it for the \MWIS problem.
Let $(G,\wei)$ be a vertex-weighted graph and let $\cM = \{M_1,\ldots,M_r\}$ be a partition of $V(G)$ into nonempty modules. Let $G/\cM$ be the corresponding quotient graph and, as before, denote the vertex of $G/\cM$ corresponding to $M_i$ by $q_i$.
Recursively, we compute a maximum-weight independent set $I_i$ in each $G[M_i]$.
Then, we define the weight function $\wei'$ on vertices of $G/\cM$ by setting $\wei'(q_i) = \wei(I_i)$ for each $i\in[r]$.
Independent sets of $G/\cM$ correspond to independent sets of $G$ that, for every $i \in [r]$, either avoid $M_i$ or use an independent set contained in $M_i$.
With the weights $\wei'$, choosing $q_i$ in the quotient graph corresponds to taking the precomputed set $I_i$ in $G[M_i]$.

This way, we can solve \MWIS in $G$ by traversing the modular decomposition tree in a bottom-up fashion and solving the problem in the quotient graph of each non-leaf node.
As the size of the modular decomposition tree is linear in the size of $G$, the total computation time is polynomial times the time needed to solve \MWIS in quotient graphs.

We remark that if the quotient graph is complete or edgeless, then the task of solving \MWIS in $G/\cM$ is trivial. Thus, the only non-trivial case left is when $G/\cM$ is prime.

The reasoning above is encapsulated in the following lemma, see e.g.,~\cite[Theorem 1]{lozin2008polynomial} for the proof.

\begin{lemma}\label{lem:modular-mwis}
Let $(G,\wei)$ be a vertex-weighted graph with $n$ vertices.
Suppose that there is an algorithm that,
given a prime induced subgraph $G'$ of $G$
and any weight function $\wei'$ on $G'$, solves \MWIS on $(G',\wei')$ in time $T(n)$.
Then \MWIS on $(G,\wei)$ can be solved in time $T(n)\cdot n^{\Oh(1)}$.
\end{lemma}

\paragraph{Split graphs and modules of $P_4$-free graphs}
A graph is \emph{split} if its vertex set can be partitioned into a clique and an independent 
set. We will make use of the following result. It was shown by Hayward, Hougardy, and Reed~\cite{DBLP:conf/soda/HaywardHR02}, but we present an alternative proof that is more consistent with the terminology of our paper.

\begin{lemma}\label{lem:HHR}
Let $G$ be a prime graph. If $G$ is not split, then every vertex of $G$ belongs to an induced $P_4$.
\end{lemma}
\input{hhr-proof.tex}

Let us explain why this result is relevant to our setting.
Consider a graph $G$ such that for some $B \subseteq V(G)$, $G$ is $(B,P_4)$-clean;
such a graph can be obtained by \cref{thm:H-clean-containers-is}.
Apply \cref{lem:modulardecomposition} to $G$ to obtain its modular decomposition tree $\mathsf{T}$.
Now, consider a non-leaf node $M$ of $\mathsf{T}$, and let $Q_M$ be the quotient graph corresponding to the partition of $G[M]$ given by its children.
Note that, either $Q_M$ is complete (and thus, split), edgeless (again, split), or prime.
In the last case, applying \cref{lem:HHR} to $Q_M$ shows that if $Q_M$ is not split, then every vertex of $Q_M$ belongs to an induced $P_4$.
However, recall that picking any vertex from each submodule corresponding to a vertex of $Q_M$, we obtain an induced subgraph of $G$ isomorphic to $Q_M$.
Consequently, if $Q_M$ is not split, then no vertex of $M$ can belong to $B$: otherwise, taking such a vertex as the representative of its child module and arbitrary representatives of the other modules of an induced $P_4$ in $Q_M$ would yield a $B$-rooted induced $P_4$ in $G$.

\paragraph{Solving \MWIS in $(B,P_4)$-clean prime graphs}
The following lemma combines the above observations with the modular decomposition tree to show that \MWIS can be solved efficiently in $(B,P_4)$-clean graphs, provided that components of $G-B$ are ``easy.''

\begin{lemma}\label{lem:terminal-p4-solver}
Let $(G,\wei)$ be a vertex-weighted graph and let $B\subseteq V(G)$
such that $G$ is $(B,P_4)$-clean.
Assume that there is an algorithm \textsf{Alg} that, given an $n'$-vertex induced subgraph of $G-B$, in time $T(n')$ computes a maximum-weight independent set in that subgraph.
Then a maximum-weight independent set in $(G,\wei)$ can be computed in time $T(n) \cdot n^{\Oh(1)}$.
\end{lemma}

\begin{proof}
By \cref{lem:modular-mwis}, it is enough to show how to solve \MWIS in prime induced subgraphs of $G$.
Thus, let $G'$ be such a graph on $n'$ vertices.
Define $B' = B \cap V(G')$. Clearly, $G'$ is $(B',P_4)$-clean.

If $G'$ is a split graph, then we can compute a maximum-weight independent set in $G'$ in polynomial time.
Indeed, in polynomial time we can find a split partition $(C,S)$ of $G'$, where $C$ is a clique and $S$ is an independent set.
Then an independent set contains at most one vertex of $C$, so we can exhaustively guess this choice, remove its neighbors, and include all remaining vertices of $S$.

So, suppose that $G'$ is not split.
By \cref{lem:HHR}, every vertex of $G'$ belongs to an induced $P_4$.
Since $G'$ is $(B',P_4)$-clean, this implies that $B'=\emptyset$, i.e., $G'$ is an induced subgraph of $G-B$.
Thus, we can apply the algorithm \textsf{Alg} to the whole graph $G'$ to compute a maximum-weight independent set in time $T(n')$.

Applying \cref{lem:modular-mwis}, we conclude that we can solve the instance $(G,\wei)$ of \MWIS in time $T(n)\cdot n^{\Oh(1)}$, as claimed.
\end{proof}

\subsection{Wrapping up}
Let us prove \cref{thm:mwisp4} for the case that $\cH = \{P_4\}$.
The proof combines \cref{thm:H-clean-containers-is} and the results developed in this section.

\begin{theorem}\label{thm:mwisp4-special}
For every fixed $k$, given a vertex-weighted graph $(G,\wei)$,
in time $n^{\Oh(\log^2 n)}$ one can either solve \MWIS in $(G,\wei)$,
or correctly report that $\treepi_{P_4}(G)>k$.
\end{theorem}
\begin{proof}
Let $k$ be a fixed non-negative integer.
Let $(G,\wei)$ be an $n$-vertex vertex-weighted graph.
The algorithm is recursive. If $n \leq 1$, then we return $V(G)$ as the maximum-weight independent set.
Thus, suppose that $n \geq 2$ and that we can solve \MWIS in any proper induced subgraph of $G$.

First, we apply \cref{thm:computedecomposition} to $G$ and $k$.
If the algorithm reports $\treepi_{P_4}(G)>k$, then we return this correct conclusion.
Otherwise, let $(T,\beta)$ be the tree decomposition returned by the algorithm.
Thus, for every  $t\in V(T)$ it holds that $\pi_{P_4}(G,\beta(t))\leq k'$, where $k' = 8k$.
Next, we apply \cref{lem:balanced-bag}  to the graph $G$ with uniform weight function and
tree decomposition $(T,\beta)$.
This way we obtain a node $t\in V(T)$ such that, with $B=\beta(t)$, we have $\pi_{P_4}(G,B)\leq k'$ and every component of $G-B$ has at most $n/2$ vertices.

We apply \cref{thm:H-clean-containers-is} to $G$ and $B$ with $\cH = \{P_4\}$,
let $\cF_{P_4}(G,B)$ be the returned family; recall that it has size $n^{\Oh(\log n)}$.

For every $(I,F) \in \cF_{P_4}(G,B)$, we will compute a maximum-weight independent set $I_F$ in $(G[F],\wei)$.
Then, we will return the set $I \cup I_F$ of maximum weight among all $(I,F) \in \cF_{P_4}(G,B)$.
Note that by the first property of \cref{thm:H-clean-containers-is}, $I \cup I_F$ is independent, and by the last property, we are guaranteed to find an optimum solution.

So, we are left with computing a maximum-weight independent set in $(G[F],\wei)$ for a given $F$ such that $G[F]$ is $(B \cap F,P_4)$-clean.
We plan to apply \cref{lem:terminal-p4-solver} to $(G[F],\wei)$ and $B \cap F$.
The only thing needed is to provide an algorithm \textsf{Alg} that can efficiently solve \MWIS in induced subgraphs of $G[F]-B$.

Consider such a subgraph $G'$. Clearly, we can handle each component of $G'$ independently.
Every such component is contained in a single component of $G-B$, and hence has at most $n/2$ vertices.
Therefore, if $F(n)$ denotes the maximum running time of our algorithm on instances with at most $n$ vertices, then \textsf{Alg} works in time $F(n/2)\cdot n^{\Oh(1)}$.
Applying \cref{lem:terminal-p4-solver} to each member of $\cF_{P_4}(G,B)$, we conclude that the overall running time is bounded by
\[
    F(n) \leq n^{\Oh(\log n)} \cdot F(n/2) \cdot n^{\Oh(1)} = n^{\Oh(\log n)} F(n/2),
\]
which is solved by $n^{\Oh(\log^2 n)}$. This completes the proof.
\end{proof}

Now, \cref{thm:mwisp4} follows from \cref{thm:mwisp4-special} combined with \cref{lem:structural-package}~\eqref{sprop:4} and \eqref{sprop:7}.

%% file: hhr-proof.tex
\begin{proof}
    Before proving the lemma, let us observe that its statement is invariant under complementation.
    Let $\overline{G}$ denote the complement of $G$, i.e., the graph on the same vertex set as $G$ where two vertices are adjacent if and only if they are not adjacent in $G$.
    Then $G$ is prime if and only if $\overline{G}$ is prime, $G$ is split if and only if $\overline{G}$ is split, and a vertex $v$ belongs to an induced $P_4$ in $G$ if and only if it belongs to an induced $P_4$ in $\overline{G}$, because $P_4$ is self-complementary.
        
    Let $G$ be a prime graph and let $v$ be a vertex that does not belong to any induced $P_4$ in $G$.
    We will show that $G$ is split.
    
    \begin{claim}\label{clm:split-is}
        Non-neighbors of $v$ form an independent set.
    \end{claim}
    \begin{claimproof}
        Let $A = N(v)$ and $B = V(G) \setminus (A \cup \{v\})$.
        Consider a component $C$ of $G[B]$.
        First, note that every vertex $a \in A$ is either complete or anticomplete to $C$.
        Indeed, suppose that $a$ has both a neighbor $b$ and a non-neighbor $b'$ in $C$.
        Since $C$ is connected, we can pick $b,b'$ so that they are adjacent.
        Then $v-a-b-b'$ is an induced $P_4$ that contains $v$, a contradiction.

        We observe that $C$ is a module in $G$.
        Indeed, $v$ is anticomplete to $C$,
        every vertex from $A$ is either complete or anticomplete to $C$, and every vertex from $B \setminus C$ is anticomplete to $C$ as $C$ is a component of $G[B]$.

        Note that $C$ must be a trivial module, as $G$ is prime.
        As $v \notin C$, we conclude that $C$ has only one vertex.
        The choice of $C$ was arbitrary, so indeed $B$ is an independent set.               
    \end{claimproof}

    Now consider the graph $\overline{G}$. Applying \cref{clm:split-is} to $\overline{G}$ and $v$, we conclude that non-neighbors of $v$ in $\overline{G}$ form an independent set in $\overline{G}$, i.e., neighbors of $v$ in $G$ form a clique in $G$.  

    Consequently, the vertex set of $G$ can be partitioned into a clique $N[v]$ and an independent set $V(G) \setminus N[v]$, i.e., $G$ is a split graph.
\end{proof}

%% file: sec-p5-3coloring.tex
This section is devoted to the proof of \cref{thm:colpfive}.

\thmcolpfive*

The proof has the same general structure as the proof of \cref{thm:mwisp4}.
The application of \cref{thm:H-clean-containers-is} for $\cH = \{P_4\}$ is replaced with \cref{thm:H-clean-containers-coloring} for $\cH = \{P_5\}$.
So now, let us discuss how to replace \cref{lem:HHR,lem:terminal-p4-solver}.

\subsection{Dealing with $(B,P_5)$-clean instances.}
We will need two classic tools from the literature.
The first one is usually referred to as the \emph{Gyárfás path argument}~\cite{gyarfas1987problems}
(see also~\cite[Lemma~5.3]{chudnovsky2024quasi} or~\cite[Lemma~3.2]{chudnovsky2025dominated}).

\begin{lemma}\label{lem:gyarfas-path}
Let $(G,\wei)$ be a connected vertex-weighted graph,
and let $v\in V(G)$.
There is an induced path $P$ starting at $v$ such that every component of $G-N[P]$ has weight at most $\wei(G)/2$.
Furthermore, $P$ can be found in polynomial time.
\end{lemma}

The second one is a result of Edwards~\cite{edwards1986complexity}, which shows that instances of list coloring where each list has size at most 2 can be solved in polynomial time by a reduction to 2-\textsc{Sat}.

\begin{lemma}\label{lem:edwards}
Let $G$ be a graph and let $L : V(G) \to 2^{\N}$ be a list assignment such that $|L(v)| \leq 2$ for every $v \in V(G)$. Then, one can check in polynomial time whether $G$ admits a proper coloring in which every vertex receives a color from its list.
\end{lemma}

Now, we can show the following lemma, inspired by the algorithm for \lcol{$k$} in $P_5$-free graphs~\cite{DBLP:journals/algorithmica/HoangKLSS10}.

\begin{lemma}\label{lem:p5clean}
    Let $(G,L)$ be an instance of \lcol{3}, where $G$ has $n$ vertices and is connected.
    Let $B \subseteq V(G)$ be a nonempty set of vertices such that $G$ is $(B,P_5)$-clean.
    Then, in polynomial time, one can either solve $(G,L)$ or enumerate a family $\cC$ of $\Oh(n^4)$ instances of $(G,L)$
    such that for each $(G',L') \in \cC$:
    \begin{enumerate}
        \item $G'$ is an induced subgraph of $G$,
        \item each component of $G'$  has at most $n/2$ vertices, and
        \item $(G,L)$ is a yes-instance of \lcol{3} if and only if there exists $(G',L') \in \cC$ such that $(G',L')$ is a yes-instance of \lcol{3}.
    \end{enumerate} 
\end{lemma}
\begin{proof}    
    Pick $x_1 \in B$ and apply \cref{lem:gyarfas-path} to $G$ with uniform weight function and vertex $x_1$
    to obtain an induced path $P = x_1-x_2-\ldots-x_t$ starting at $x_1$ such that every component of $G-N[P]$ has at most $n/2$ vertices.
    Since $G$ is $(B,P_5)$-clean, we have that $t \leq 4$. For simplicity of exposition, assume that $t=4$; if $t < 4$, some objects discussed in the proof do not exist. We emphasize that their existence is never used in the proof, and thus, the argument is still valid.

    Let $N = N(P)$ and for $i \in [4]$, let $N_i = N \cap N(x_i) \setminus \bigcup_{j < i} N[x_j]$.

    Let $R = V(G) - N[P]$.
    By the choice of $P$, every component of $G[R]$ has at most $n/2$ vertices.

\paragraph{Edges between $N$ and $R$.}

    Let us analyze the neighborhood of $N$ in $R$. We start with sets $N_2,N_3$, and $N_4$.
    \begin{claim}
        The following hold:
        \begin{enumerate}
            \item for every $v \in N_2$ and every component $C$ of $G[R]$, either $v$ is complete to $C$ or anticomplete to $C$,
            \item there are no edges between $N_3$ and $R$,
            \item $N_4 = \emptyset$.
        \end{enumerate}
    \end{claim}
    \begin{claimproof}
        (1) Suppose there is a vertex $v \in N_2$ and a component $C$ of $G[R]$ such that $v$ has a neighbor and a non-neighbor in $C$. Since $G[C]$ is connected, there must be two adjacent vertices $u,u' \in C$ such that $v$ is adjacent to $u$ and non-adjacent to $u'$. Then, $x_1-x_2-v-u-u'$ is an induced $P_5$ intersecting $B$, a contradiction.

        \noindent (2) Suppose there is an edge $vu$ where $v \in N_3$ and $u \in R$. Then, $x_1-x_2-x_3-v-u$ is an induced $P_5$ intersecting $B$, a contradiction.

        \noindent (3) Suppose there is a vertex $v \in N_4$. Then, $x_1-x_2-x_3-x_4-v$ is an induced $P_5$ intersecting $B$, a contradiction.
    \end{claimproof}

    Now let us analyze the neighborhood of $N_1$ in $R$.
    Suppose that $(G,L)$ is a yes-instance of \lcol{3} and let $\phi$ be a proper coloring of $G$ respecting $L$. 
    For each $i \in [3] \setminus \{\phi(x_1)\}$, let $N_1^i$ be the set of vertices $v \in N_1$ such that $\phi(v) = i$.
    Let $S^i$ be an inclusion-wise minimal subset of $N_1^i$ with the following
property: for every vertex $u\in R$, if $u$ has a neighbor in $N_1^i$, then
$u$ has a neighbor in $S^i$.
    
    \begin{claim}
        For each $i \in [3] \setminus \{\phi(x_1)\}$, we have $|S^i| \leq 2$.
    \end{claim}
    \begin{claimproof}        
        Suppose that $|S^i| \geq 3$ for some $i \in [3] \setminus \{\phi(x_1)\}$. Let $v_1,v_2,v_3 \in S^i$ be distinct vertices.
        Note that $S^i$ is independent, since they all receive the same color in $\phi$.
        By minimality of $S^i$, there are three distinct vertices $v'_1,v'_2,v'_3$ in $R$ such that $v'_j$ is adjacent to $v_j$ and non-adjacent to $S^i \setminus \{v_j\}$ for each $j \in [3]$.

        If there is a non-edge between $v'_1,v'_2,v'_3$, say, $v_1'v'_2 \notin E(G)$,
        then $v'_1-v_1-x_1-v_2-v'_2$ is an induced $P_5$ intersecting $B$.
        Otherwise, $v'_1,v'_2,v'_3$ form a triangle. Then $v'_1-v'_2-v_2-x_1-v_3$ is an induced $P_5$ intersecting $B$.
        In both cases, we obtain a contradiction.
    \end{claimproof}

    Now, we exhaustively guess the coloring $\phi$ of $P$ and the sets $S^i$ for each $i \in [3] \setminus \phi(x_1)$ corresponding to $\phi$. This results in at most $3 \cdot 3^3 \cdot n^4 = \Oh(n^4)$ branches.
    For each branch, we proceed as follows.
    For each vertex whose color is guessed, we remove all colors from its list except for the guessed color.
    For each $u \in R$, if $u$ has no neighbor in $S^i$, then we remove color $i$ from the list of every neighbor of $u$ in $N_1$.
    Note that these operations are consistent with the coloring $\phi$ and the definition of $S^i$.

    Then, we exhaustively apply the following reduction rules:
    \begin{enumerate}
        \item If there is a vertex $v$ such that $L(v) = \emptyset$, then we discard the current branch.
        \item If there is a vertex $v$ such that $|L(v)| = 1$, then we remove the unique color from the lists of all neighbors of $v$ and remove $v$ from the graph.
    \end{enumerate}
        
    \begin{claim}
        There are no edges between $N_1$ and $R$.
    \end{claim}
    \begin{claimproof}
        Suppose there is an edge $vu$ where $v \in N_1$ and $u \in R$.
        As reduction rules cannot be applied, we know that $|L(v)| \geq 2$ and $|L(u)| \geq 2$.
        Consequently, there is a color $i \in L(v) \cap L(u)$.

        If $i = \phi(x_1)$, then the color $i$ was removed from $L(v)$, a contradiction.
        Otherwise, $i \in [3] \setminus \{\phi(x_1)\}$.
        If $u$ is adjacent to a vertex from $S^i$, then the color $i$ was removed from $L(u)$, a contradiction.
        Thus, $u$ has no neighbor in $S^i$. Then, the color $i$ was removed from $L(v)$, a contradiction.
    \end{claimproof}

    \paragraph{Components of $G$.}
    Let us summarize the properties of the current instance.
    There are three possible types of components $G'$ of $G$:
    \begin{enumerate}
        \item $G'$ is contained in $R$; this means that it has at most $n/2$ vertices by the properties of $P$,
        \item $G'$ is contained in $N$; this means that every vertex of $G'$ is adjacent to some vertex of $P$ and thus, the list of such a vertex has size 2,
        \item $G'$ intersects both $N$ and $R$. 
    \end{enumerate}
    Components of the first type are consistent with the properties of $\cC$ in the statement of the lemma.
    Components of the second type can be solved in polynomial time by \cref{lem:edwards}.
    
    Thus, let us consider a component $G'$ of the third type.
    By symmetry, assume that the color of $x_2$ is 3 and thus, every vertex in $N_2$ has list $\{1,2\}$
    (we excluded the color of $x_2$ and removed vertices with smaller lists). 

    Let $D_1,\ldots,D_p$ be the components of $G'[R]$. Note that the neighborhood of each $D_i$ is contained in $N_2$,
    and every vertex $v \in N_2$ that is adjacent to $D_i$ is actually complete to $D_i$.
    Thus, for each $D_i$, the colors used on $D_i$ will be either a subset of $\{1,3\}$ or of $\{2,3\}$.

    Consider each $D_i$ separately.
    If $D_i$ consists of a single vertex with 3 in its list, we remove this vertex.
    Otherwise, for $j \in \{1,2\}$, we check whether $D_i$ can be colored with colors  $\{j,3\}$, additionally respecting the lists. This can be done in polynomial time by \cref{lem:edwards}.
    We remove $D_i$ from the graph and instead introduce a new vertex $d_i$ adjacent to all vertices of $N_2$ that are adjacent to $D_i$ (this can be equivalently seen as contracting $D_i$ to $d_i$).
    The list of $d_i$ is set to $\{ j \in \{1,2\} ~|~ D_i$ can be colored with colors $\{j,3\}$ while respecting the lists$\}$.

    Denote the resulting instance as $(G^*,L^*)$.

    \begin{claim}\label{clm:from-old-to-new}
        If $(G',L)$ is a yes-instance of \lcol{3}, then $(G^*,L^*)$ is a yes-instance of \lcol{3}.
    \end{claim}
    \begin{claimproof}
        Let $\phi$ be a proper coloring of $G'$ respecting $L$.
        We will define a proper coloring $\phi^*$ of $G^*$ respecting $L^*$.
        On vertices of $G^*$ that are also in $G'$, we set $\phi^*(v) = \phi(v)$.
        Thus, we only need to define colors for vertices $d_i$.
                        
        Recall that for each vertex $v \in N_2$, we have $\phi(v) \in \{1,2\}$.        
        Consider a component $D_i$ of $G'[R]$ for which $d_i$ exists.
        Let $A$ be the set of neighbors of $D_i$ in $N$;
        recall that $A \neq \emptyset$ and every vertex of $A$ is complete to $D_i$.

        Note that the set of colors used on $D_i$ is either a subset of $\{1,3\}$ or of $\{2,3\}$,
        as $\phi$ uses at least one color from $\{1,2\}$ on $A$.
        By symmetry, assume that the former holds. This means $1 \in L^*(d_i)$.
                
        We observe that color $1$ appears on some vertex of $D_i$.
        Otherwise, if $\phi$ uses only color $3$ on $D_i$, then $D_i$ is a single vertex with 3 in its list, 
        and thus, $d_i$ does not exist.
        Consequently, no vertex in $A$ is colored 1, and thus, we can safely set $\phi^*(d_i)=1$.
        
        Repeating this for every $D_i$, we obtain a proper coloring $\phi^*$ of $G^*$ respecting $L^*$.
    \end{claimproof}

        \begin{claim}\label{clm:from-new-to-old}
        If $(G^*,L^*)$ is a yes-instance of \lcol{3}, then $(G',L)$ is a yes-instance of \lcol{3}.
    \end{claim}
    \begin{claimproof}
        Suppose that $\phi^*$ is a proper coloring of $G^*$ respecting $L^*$.
        We will define a proper coloring $\phi$ of $G'$ respecting $L$.
        On vertices of $G'$ that are also in $G^*$, we set $\phi(v) = \phi^*(v)$.
        Thus, we only need to define colors for vertices of $D_i$.

        Again, note that for each vertex $v \in N_2$, we have $\phi(v) \in \{1,2\}$.
        Consider a component $D_i$ of $G'[R]$.

        If $D_i$ is a single vertex with 3 in its list, then we can set $\phi(v) = 3$ for the unique vertex $v \in D_i$. Since every neighbor of $v$ lies in $N_2$ and is colored either $1$ or $2$,
        assigning color $3$ to $v$ does not violate properness.

        So, suppose that this is not the case.
        This means that $d_i$ exists and $\phi^*(d_i) \in \{1,2\}$.
        By symmetry, assume that $\phi^*(d_i) = 1$.
        This means that all neighbors of $d_i$ in $G^*$ are colored $2$.

        Since $1$ was included in $L^*(d_i)$, there is a proper coloring $\phi_{i,1}$ of $D_i$ with colors $\{1,3\}$ respecting the lists.
        We can set $\phi(v) = \phi_{i,1}(v)$ for every $v \in D_i$.
        
        Repeating this for every $D_i$, we obtain a proper coloring $\phi$ of $G'$ respecting $L$.
    \end{claimproof}

    Combining \cref{clm:from-old-to-new} with \cref{clm:from-new-to-old}, we conclude that $(G',L)$ is a yes-instance of \lcol{3} if and only if $(G^*,L^*)$ is a yes-instance of \lcol{3}.
    On the other hand, every vertex of $G^*$ has list of size at most 2, and thus,
    we can solve $(G^*,L^*)$ in polynomial time by \cref{lem:edwards}.
    
  \paragraph{Summarizing.}
  So let us go back to the components of $G$ that we analyzed.
  Components of the second type can be solved in polynomial time by \cref{lem:edwards}.
  Components of the third type can also be solved in polynomial time using the reduction above and again calling \cref{lem:edwards}.
  If, for any of these components we conclude that it is a no-instance, then we can discard the current branch.
  Otherwise, we can remove all components of the second and third type from $G$ and keep only components of the first type, which have at most $n/2$ vertices.
  Then, we include the resulting instance in $\cC$.

  The number of branches is $\Oh(n^4)$, and the work done in each branch is polynomial.
  This completes the proof of the lemma.
\end{proof}

\subsection{Wrapping up}
Again, let us first show \cref{thm:colpfive} for the case that $\cH = \{P_5\}$.

\begin{theorem}\label{thm:colpfive-special}
For every fixed $k$, given an instance $(G,L)$ of \lcol{3},
in time $n^{\Oh(\log^2 n)}$ one can either solve $(G,L)$,
or correctly report that $\treepi_{P_5}(G)>k$.
\end{theorem}

\begin{proof}
    Let $(G,L)$ be an $n$-vertex instance of \lcol{3} such that $\treepi_{P_5}(G)\leq k$.
    The algorithm is recursive. If $n \leq 1$, then we can solve the instance by brute force.
    Thus, assume that $n \geq 2$ and we can solve every instance with $n' < n$ vertices in time $F(n')$.
    We aim to show that we can solve $(G,L)$ in time $F(n) = 2^{\Oh(\log^3 n)}$.

    We apply \cref{thm:computedecomposition} to $G$ with $\cH = \{P_5\}$.
    If this application reports that $\treepi_{P_5}(G)>k$, then we are done.
    Otherwise, we obtain a tree decomposition $(T,\beta)$ of $G$ such that $\pi_{P_5}(G,\beta(t)) \leq 8k$ for every $t \in V(T)$.
    We apply \cref{lem:balanced-bag} to $G$ with uniform weight function and $(T,\beta)$ to obtain a node $t \in V(T)$ such that, with $B = \beta(t)$, we have $\pi_{P_5}(G,B) \leq 8k$ and every component of $G-B$ has at most $n/2$ vertices.

    We apply \cref{thm:H-clean-containers-coloring} to $(G,L)$ and $B$ with $\cH = \{P_5\}$. If this application solves $(G,L)$, we are done. 
    So suppose otherwise and let $\cF_{P_5}(G,B)$ be the obtained family of instances.
    By the last property in~\Cref{thm:H-clean-containers-coloring}, we need to verify whether there is $(G',L') \in \cF_{P_5}(G,B)$ such that  $(G',L')$ is a yes-instance of \lcol{3}.
    Thus, for every $(G',L') \in \cF_{P_5}(G,B)$, we proceed as follows.
    Clearly, the colorability of $(G',L')$ can be checked independently in each component of $G'$.
    Thus, we can assume that $G'$ is connected.

    If $B \cap V(G') = \emptyset$, then $G'$ is contained in a component of $G-B$, and thus, $G'$ has at most $n/2$ vertices. In this case we call the algorithm recursively on $(G',L')$.  
    Otherwise, we invoke \cref{lem:p5clean} for $(G',L')$ and $B \cap V(G')$.
    If this invocation solves $(G',L')$, we use the returned answer. Otherwise,
    let $\cC$ be the family of instances returned by this application.

    In order to solve $(G',L')$, we need to check whether there is $(G'',L'') \in \cC$ that is a yes-instance of \lcol{3}. 
    By the properties of $\cC$ in \cref{lem:p5clean}, every component of such $G''$ is an induced subgraph of $G'$ (and thus, of $G$) and has at most $|V(G')|/2 \leq n/2$ vertices, so it
    can be solved recursively.

    Summing up, the running time is bounded by 
    \[
        F(n) \leq 2^{\Oh(\log^2 n)} \cdot n^{\Oh(1)} \cdot F(n/2)  = n^{\Oh(\log^2 n)},
    \]
    where the $n^{\Oh(1)}$-factor accounts for the total number of recursive calls.
    This completes the proof.
\end{proof}

Now, \cref{thm:colpfive} follows from \cref{thm:colpfive-special}  combined with \cref{lem:structural-package}~\eqref{sprop:4} and \eqref{sprop:7}.

%% file: sec-p3.tex
In this section we show that several problems can be solved in quasipolynomial time in graphs of bounded induced-$P_3$-packing treewidth. 
The following simple lemma allows us to deal with $(B,P_3)$-clean graphs.

\begin{lemma}\label{lem:p3clean}
    Let $G$ be a connected graph and let $B \subseteq V(G)$ be a nonempty set of vertices such that $G$ is $(B,P_3)$-clean. Then, $G$ is a complete graph.
\end{lemma}
\begin{proof}
    Pick any $v \in B$.
    If there is a vertex $u \in V(G)$ that is not adjacent to $v$, then the shortest path from $v$ to $u$ has at least three vertices, and thus there is an induced $P_3$ intersecting $B$, a contradiction.
    So, $v$ is adjacent to every vertex in $G$.
    If $G$ is not complete, then there are two vertices $x,y \in V(G)$ that are not adjacent.
    Then $x-v-y$ is an induced $P_3$ intersecting $B$, a contradiction.
\end{proof}


\subsection{\textsc{Odd Cycle Transversal}}
\input{sec-octp3}

\subsection{Blob graphs and their applications}\label{sec:blob}

\input{sec-blob}

%% file: sec-octp3.tex
In this section, we prove \cref{thm:octpthree}.
\thmoctpthree*

Actually, we will solve a slightly more general problem.
In \textsc{Max List $2$-Colorable Induced Subgraph}, we are given a graph $G$,
a list function $L : V(G) \to 2^{\{1,2\}}$,
and a revenue function $\rev: V(G) \times \{1,2\} \to \mathbb{Q}_{\ge 0}$.
The goal is to find two disjoint independent sets $X_1$ and $X_2$ in $G$,
such that, for every $v \in V(G)$, $v \in X_i$ implies $i \in L(v)$,
in order to maximize the revenue defined as 
\[
        \sum_{v \in X_1} \rev(v,1) + \sum_{v \in X_2} \rev(v,2).
\]
We remark that the lists $L(v)$ are actually not needed as forbidding a vertex $v$ from $X_i$
can be achieved by setting $\rev(v,i) = 0$. However, it will be convenient to keep the lists in the problem definition.

Clearly, \textsc{Max List $2$-Colorable Induced Subgraph} generalizes \OCT.
Indeed, if $L(v) = \{1,2\}$ and $\rev(v,1) = \rev(v,2) = \wei(v)$ for all $v \in V(G)$,
then the maximum revenue of a solution to \textsc{Max List $2$-Colorable Induced Subgraph}
is equal to the sum of weights of all vertices minus the weight of a minimum-weight odd cycle transversal.

Now, let us proceed to the proof.
Again, by \cref{lem:structural-package}~\eqref{sprop:4} and \eqref{sprop:7},
it is enough to assume that $\cH = \{P_3\}$.

\begin{theorem}
For every fixed $k$, given an $n$-vertex instance $(G,L,\rev)$ of \textsc{Max List $2$-Colorable Induced Subgraph},
in time $n^{\Oh(\log^2 n)}$ one can either solve it 
or correctly report that $\treepi_{P_3}(G)>k$.
\end{theorem}

\begin{proof}
    Let $(G,L,\rev)$ be an instance of \textsc{Max List $2$-Colorable Induced Subgraph} with $n$ vertices.
    The algorithm is recursive. If $n \leq 1$, then the problem can be trivially solved by brute force.
    Thus, suppose that $n \geq 2$ and that we can solve every instance with $n' < n$ vertices in time $F(n')$.

    Define $k'=8k$. Similarly as before, we first apply \cref{thm:computedecomposition} and then \cref{lem:balanced-bag} to either conclude that $\treepi_{P_3}(G) > k$ or obtain a set $B \subseteq V(G)$ such that $\pi_{P_3}(G,B) \leq k'$ and every component of $G-B$ has at most $n/2$ vertices.

    Then, we apply \cref{thm:H-clean-containers-is} to obtain a family $\cF_{P_3}(G,B)$ of size $n^{\Oh(\log n)}$ with the properties listed in the lemma statement.
    In particular, if $(X_1,X_2)$ is an (unknown) optimum solution, then for each $i \in \{1,2\}$ there is $(I_i,F_i) \in \cF_{P_3}(G,B)$ such that $I_i \subseteq X_i \subseteq I_i \cup F_i$.

    \paragraph{General plan.}
    For every pair $(I_1,F_1),(I_2,F_2) \in \cF_{P_3}(G,B)$, we proceed as follows.
    First, if $I_1 \cap I_2 \neq \emptyset$, or $I_i$ contains a vertex $v$ such that $i \notin L(v)$, then we discard this pair.
    Next, we set $L'(v) = L(v) \cap \{ i ~|~ v \in F_i\}$ for every $v \in F_1 \cup F_2$.
    Consider the instance $(G',L',\rev)$, where $G'=G[(F_1 \cup F_2) \setminus (I_1 \cup I_2)]$.

    In each of such instances, we will find a maximum-revenue solution $(X_1',X_2')$.
    Note that $(X_1' \cup I_1, X_2' \cup I_2)$ is a feasible solution to the original instance $(G,L,\rev)$. Indeed, $I_1$ and $I_2$ are independent sets, the lists are respected by construction, and every vertex of $X_i'$ belongs to $F_i \subseteq V(G) \setminus N[I_i]$, so $X_i'$ has no neighbor in $I_i$ and is disjoint from $I_{3-i}$.
    Thus, the total revenue of the solution corresponding to the pair $(I_1,F_1),(I_2,F_2)$ is obtained by summing the revenue of $(X_1',X_2')$ and the revenue of vertices in $I_1 \cup I_2$.
    We will return the solution of largest revenue; clearly, it will be an optimum solution to the original instance.

    \paragraph{Cleaning.}
    Thus, let us focus on solving one instance $(G',L',\rev)$ for a fixed pair $(I_1,F_1),(I_2,F_2)$.
    Define $B' = B \cap V(G')$.
    Let us introduce two straightforward cleaning steps.
    \begin{enumerate}
        \item If there is a vertex with an empty list, then we remove it from $G'$.
        \item If, for some $v \in V(G')$ and $i \in \{1,2\}$, it holds that $\rev(v,i) = 0$, then we remove $i$ from $L'(v)$.
    \end{enumerate}
    
    We apply these cleaning steps exhaustively; they clearly can be performed in polynomial time, and they do not change the optimum revenue of the instance.
    For simplicity we keep the notation $(G',L',\rev)$ for the resulting instance, and $B' = B \cap V(G')$ for the current $G'$.
    For $i \in \{1,2\}$, let $V_i = \{v \in V(G') ~|~ i \in L'(v)\}$.
    Note that $V_i \subseteq F_i$ and thus, $G'[V_i]$ has no induced $P_3$ intersecting $B' \cap V_i$.

    \paragraph{Reducing $i$-simplicial vertices.}
    For $i \in \{1,2\}$, we say that a vertex $v$ is \emph{$i$-simplicial} if $v \in V_i$ and the neighborhood of $v$ in $G'[V_i]$ is a clique.

    \begin{claim}\label{clm:simplicial}
        For $i \in \{1,2\}$, every vertex in $B' \cap V_i$ is $i$-simplicial.
    \end{claim}
    \begin{claimproof}
        Let $v \in B' \cap V_i$ and define $N = N[v] \cap V_i$.
        The graph $G'[N]$ is connected and, as induced subgraph of $G'[V_i]$, it is $(B'\cap N,P_3)$-clean.
        Thus, by \cref{lem:p3clean}, $G'[N]$ is a clique.
    \end{claimproof}

    Now, let us introduce a reduction rule that allows us to deal with $i$-simplicial vertices.
    \begin{description}
        \item[Reduction Rule] Let $v \in B'$ be a vertex with $L'(v) = \{i\}$.
        Let $N^+$ (resp., $N^-$) be the set of neighbors $u$ of $v$ in $V_i$ with $\rev(u,i) \geq \rev(v,i)$ (resp., $\rev(u,i) < \rev(v,i)$).
        Delete $i$ from lists of all vertices in $\{v\} \cup N^-$ and decrease $\rev(u,i)$ for every vertex $u \in N^+$ by $\rev(v,i)$. Then, exhaustively apply the cleaning steps.
    \end{description}

    Clearly, the reduction rule can be performed in polynomial time.
    Note that the cleaning steps in particular remove $v$ from the graph.
    Now, let us argue that the reduction rule is safe in the following sense.

    \begin{claim}\label{clm:reduction-rule}
        Let $\OPT$ be the optimum revenue for the instance before reduction and let $\OPT'$ be the optimum revenue for the instance after reduction. Then, $\OPT = \OPT' + \rev(v,i)$.       
    \end{claim}
    \begin{claimproof}
    By symmetry, let $i=1$.
    By \cref{clm:simplicial}, $v$ is $1$-simplicial, and thus
    $N^- \cup N^+$ is a clique in $G'[V_1]$.

    We first prove that $\OPT \leq \OPT' + \rev(v,1)$.
    Let $(X_1,X_2)$ be an optimum solution of the original instance, of revenue $\OPT$.

    If $v\in X_1$, then $(X_1\setminus\{v\},X_2)$ is feasible for the
    reduced instance and has revenue $\OPT-\rev(v,1)$.
    Suppose now that $v\notin X_1\cup X_2$.
    Since $\rev(v,1)>0$, the set $X_1$ must contain a neighbor of $v$,
    because otherwise we could add $v$ to $X_1$ and obtain a better
    solution. As the neighborhood of $v$ in $V_1$ is a clique, there is
    exactly one neighbor $u$ of $v$ in $X_1$.
    If $u\in N^-$, then replacing $u$ by $v$ in $X_1$ yields a feasible
    solution of larger revenue, a contradiction. Hence $u\in N^+$.
    Therefore, $(X_1,X_2)$ is feasible for the reduced instance, and
    its revenue there is exactly $\OPT-\rev(v,1)$, because only the value
    of $u$ changes. Consequently,
    \[
        \OPT' \geq \OPT-\rev(v,1),
    \]
    and therefore $\OPT \leq \OPT'+\rev(v,1)$.

    We now prove that $\OPT \geq \OPT' + \rev(v,1)$.
    Let $(X'_1,X'_2)$ be an optimum solution of the reduced instance.

    If $X'_1$ contains a neighbor $u$ of $v$, then necessarily $u\in N^+$.
    Hence $(X'_1,X'_2)$ is feasible for the original instance as well, and
    its revenue increases by exactly $\rev(v,1)$ when we go back to the
    original instance, because the contribution of $u$ is restored.
    Thus, in this case, the original instance has a feasible solution of
    revenue $\OPT'+\rev(v,1)$.

    Otherwise, $X'_1$ contains no neighbor of $v$.
    Since $L'(v)=\{1\}$ in the original instance, the pair
    $(X'_1\cup\{v\},X'_2)$ is feasible for the original instance, and its
    revenue is again $\OPT'+\rev(v,1)$.

    This completes the proof of claim.
    \end{claimproof}

    So, from now on, assume that the reduction rule is applied exhaustively;
    we keep denoting the instance $(G',L',\rev)$.
    We note that following the proof of \cref{clm:reduction-rule}, we can turn an optimum solution of the reduced instance into an optimum solution of the instance before reduction rule was applied.

 \paragraph{Layer-connected components.}
    Next, similarly as before, we observe that we can split the current instance into independent instances, each corresponding to a component of $G'$. However, we can further refine this splitting.

    An edge $uv$ of $G'$ is called \emph{crossing} if $u \in V_1 \setminus V_2$ and $v \in V_2 \setminus V_1$, or vice versa. If $uv$ is not crossing, i.e., there is $i \in \{1,2\}$ such that $u,v \in V_i$, then it is a \emph{layer edge}.
    Note that crossing edges are not relevant for the problem as lists already forbid putting both $u$ and $v$ into the same independent set.

    A \emph{layer-connected component} of a graph is a maximal induced subgraph where any two vertices are connected by a path consisting of layer edges only. Note that each layer-connected component is in particular connected, and it might contain crossing edges.
    Furthermore, layer-connected components can be computed in polynomial time by finding components of the graph obtained from $G'$ by removing all crossing edges.

    For each layer-connected component of $G'$, we can solve the problem independently, and the optimum revenue of the original instance is equal to the sum of optimum revenues of all components.
    We remark that even though removing edges might create many $P_3$s, each considered subinstance is still an induced subgraph of $G'$, and thus, it is still $(B',P_3)$-clean.

    So, from now on, we can assume that $G'$ is layer-connected.

    \paragraph{Solving reduced layer-connected instances.}
    Note that if $B' = \emptyset$, then $G'$ is contained in a component of $G-B$ and thus has at most $n/2$ vertices. In this case, we can solve the instance $(G',L',\rev)$ recursively in time $F(n/2)$.
    
    Thus, suppose that $B' \neq \emptyset$.
    It turns out that after applying the reduction rule and splitting into layer-connected components, the structure of $G'$ is very simple.

    \begin{claim}\label{clm:twocliques}
        If $B' \neq \emptyset$, then $G'[V_1]$ and $G'[V_2]$ are cliques.
    \end{claim}
    \begin{claimproof}
        Note that $B' \subseteq V_1 \cap V_2$, as all vertices in $B'$ with a singleton list have been removed by the reduction rule.        
        Pick any $v \in B'$. Thus $v \in V_1 \cap V_2$.
        For $i \in \{1,2\}$, let $K_i$ be the vertex set of the component of $G'[V_i]$ containing $v$.
        By \cref{lem:p3clean}, both $K_1$ and $K_2$ induce cliques in $G'$.

        We claim that $V(G') = K_1 \cup K_2$.
        Suppose for contradiction that there is a vertex outside $K_1 \cup K_2$.
        As $G'$ is layer-connected, there is a path $P$ from $v$ to such a vertex consisting only of layer edges.
        Let $yz$ be the first edge of $P$ such that $y \in K_1 \cup K_2$ and $z \notin K_1 \cup K_2$.
        By symmetry, we may assume that $y,z \in V_1$ (here we use that $yz$ is a layer edge). We claim that $y \in K_1$.

        Suppose that $y \in K_2 \setminus K_1$.
        Since $v \in K_2$ and $K_2$ is a clique, we have $yv \in E(G')$.
        Now, since $v \in V_1$ and $y \in V_1$, they must be in the same component of $G'[V_1]$.
        Consequently, $y$ belongs to $K_1$, a contradiction.

        Thus, we indeed have $y \in K_1$.
        Since $y,z \in V_1$ and $yz \in E(G')$, the vertex $z$ also belongs to the component of $G'[V_1]$ containing $v$, namely $K_1$, a contradiction.

        Consequently, $V(G')$ has no vertex outside $K_1 \cup K_2$.
        To finish the proof, suppose there is some $u \in K_2 \cap V_1 \setminus K_1$ (the other case is symmetric).
        Since $u \in K_2$ and $K_2$ is a clique, there is an edge $uv$ in $G'$.
        Since $u,v \in V_1$, they must be in the same component of $G'[V_1]$, and thus $u \in K_1$, a contradiction.
    \end{claimproof}

        By \cref{clm:twocliques}, any reduced instance can be solved very easily, as any solution will contain at most one vertex from $V_1$ and at most one vertex from $V_2$. Such an instance can be solved in polynomial time by brute force.

        \paragraph{Running time.}
        Summing up, after applying cleaning, the reduction rule exhaustively, and splitting into layer-connected components,
        every layer-connected instance results in an instance that can be either solved in time $F(n/2)$ by applying recursion, or in polynomial time using \cref{clm:twocliques}.
        As the total number of instances considered for one pair of elements of $\cF_{P_3}(G,B)$ is polynomial, the number of leaves in the recursion tree is bounded by:
        \[
        F(n) \leq \left( n^{\Oh(\log n)} \right)^2 \cdot n^{\Oh(1)} \cdot F(n/2) = n^{\Oh(\log^2 n)}.
        \]
        As all local computation can be performed in polynomial time, the overall running time is $n^{\Oh(\log^2 n)}$, as claimed.
\end{proof}

%% file: sec-blob.tex
In this section we show that the property of having bounded induced-$P_3$-packing treewidth is preserved under the so-called \emph{blob graph} operation. This allows us to solve several problems in quasipolynomial time in graphs of bounded induced-$P_3$-packing treewidth, by reducing them to corresponding instances of \MWIS.

Let $G$ be a graph. The \emph{blob graph} of $G$, denoted $G^\circ$, is
the graph defined as follows:
\begin{align*}
       V(G^\circ) = & \{X\subseteq V(G) ~|~ X\ne\emptyset\text{ and } G[X]\text{ is connected}\},\\
       E(G^\circ) = & \{XY ~|~ X\cap Y\neq\emptyset \text{ or there is an edge of } G \text{ between } X \text{ and } Y\}.
\end{align*}
In other words, $XY$ is an edge if and only if $G[X \cup Y]$ is connected.

\begin{lemma}\label{lem:P3-blob-preservation}
For every graph $G$ it holds that $\treepi_{P_3}(G) = \treepi_{P_3}(G^\circ)$.  
\end{lemma}

\begin{proof}
Note that $G$ is an induced subgraph of $G^\circ$ (induced by singleton sets),
and therefore $\treepi_{P_3}(G)\leq \treepi_{P_3}(G^\circ)$.
It remains to prove the reverse inequality.

Let $\cT = (T,\beta)$ be a tree decomposition of $G$ with $\pi_{P_3}(G,\beta(t)) \leq k$
for every $t \in V(T)$, where $k = \treepi_{P_3}(G)$.
We define a tree decomposition $\cT^\circ = (T,\beta^\circ)$ of $G^\circ$ as follows.
For every node $t \in V(T)$, define
\[
    \beta^\circ(t) = \{X \in V(G^\circ) ~|~ X \cap \beta(t) \neq \emptyset\}.
\]
It is straightforward to verify that $\cT^\circ$ (see e.g.~\cite{DBLP:conf/esa/LimaMMORS24}) is a tree decomposition of $G^\circ$.

For contradiction, suppose that there is a node $t \in V(T)$ such that for $B^\circ = \beta^\circ(t)$,
it holds that $\pi_{P_3}(G^\circ,B^\circ) = k' > k$.
Let $\{P^1,\ldots,P^{k'}\}$ be a $B^\circ$-rooted induced $P_3$-packing in $G^\circ$.

Consider one such path $P^i$ and denote its consecutive vertices by $X^i_1-X^i_2-X^i_3$.

\begin{claim}
    The graph $G[X^i_1 \cup X^i_2 \cup X^i_3]$ is connected and not complete.    
\end{claim}
\begin{claimproof}
    The connectedness follows directly from the definition of edges of the blob graph.
    On the other hand, $X^i_1$ and $X^i_3$ are nonadjacent in $G^\circ$,
    and therefore they are disjoint and anticomplete in $G$.
    Thus, $G[X^i_1 \cup X^i_2 \cup X^i_3]$ is not complete.
\end{claimproof}

\begin{claim}
    Every vertex $v \in X^i_1 \cup X^i_2 \cup X^i_3$ belongs to an induced $P_3$ in $G[X^i_1 \cup X^i_2 \cup X^i_3]$.
\end{claim}
\begin{claimproof}
    Pick any $v \in X^i_1 \cup X^i_2 \cup X^i_3$ and  suppose that it does not belong to any induced $P_3$ in $G[X^i_1 \cup X^i_2 \cup X^i_3]$.
    This means that $G[X^i_1 \cup X^i_2 \cup X^i_3]$ is connected and $(\{v\},P_3)$-clean.
    By \cref{lem:p3clean}, this means that $G[X^i_1 \cup X^i_2 \cup X^i_3]$ is complete, a contradiction.
\end{claimproof}

Observe that for every $i \in [k']$, there is a vertex $v_i \in X^i_1 \cup X^i_2 \cup X^i_3$ that belongs to $\beta(t)$. Applying the previous claim, we conclude that there is an induced $P_3$ in $G[X^i_1 \cup X^i_2 \cup X^i_3]$ that intersects $\beta(t)$.
As paths $P^1,\ldots,P^{k'}$ are pairwise anticomplete in $G^\circ$, the corresponding induced $P_3$'s in $G$ are also pairwise anticomplete.
This means that there is a $\beta(t)$-rooted induced $P_3$-packing of size $k' > k$ in $G$, a contradiction.
\end{proof}

\paragraph{Algorithmic consequences.}

\cref{lem:P3-blob-preservation} allows us to reduce several problems to solving \MWIS in induced subgraphs of blob graphs.
For example, for every fixed finite family $\cH$ of connected graphs, the problem of finding a largest induced $\cH$-packing in a graph $G$ can be reduced to \MWIS in the subgraph of $G^\circ$ induced by vertices corresponding to connected induced subgraphs of $G$ that are isomorphic to a graph in $\cH$.
As the number of vertices in this induced subgraph of the blob graph is polynomial in $|V(G)|$ (as $\cH$ is fixed), the overall running time is quasipolynomial in $|V(G)|$.
See~\cite{DBLP:conf/esa/LimaMMORS24} for more details on this reduction.

\begin{corollary}
    For every fixed $k$ and every fixed finite family $\cH$ of connected graphs,
    the problem of finding a largest induced $\cH$-packing in a graph with $\treepi_{P_3} \leq k$ can be solved in quasipolynomial time.
\end{corollary}

Another example is the \textsc{Distance-$d$ Independent Set} problem where, for fixed integer $d \geq 2$, we ask for a maximum-weight set of vertices at pairwise distance at least $d$.
If $d$ is even, this problem can be reduced to \MWIS in the subgraph of $G^\circ$ induced by vertices corresponding to balls of radius $d/2-1$ around each vertex in $G$.
Again, the details can be found in~\cite{DBLP:conf/esa/LimaMMORS24}.

\begin{corollary}
    For every fixed $k$ and even integer $d \geq 2$,
    the \textsc{Distance-$d$ Independent Set} problem in graphs with $\treepi_{P_3} \leq k$ can be solved in quasipolynomial time.
\end{corollary}

Finally, \cref{lem:P3-blob-preservation} can be used to find approximate solutions for some slightly restricted version of unweighted $(\tw \leq r,\psi)$-\MWIS.
We say that a \textsf{CMSO}$_2$ formula $\psi$ is \emph{hereditary} if:
($i$) if $G \models \psi$, then $G' \models \psi$ for every induced subgraph $G'$ of $G$, and
($ii$) if $G_1 \models \psi$ and $G_2 \models \psi$, then $G_1 + G_2 \models \psi$, where $G_1 + G_2$ is the disjoint union of $G_1$ and $G_2$.
We note that many natural graph properties, like e.g., planarity, bounded degeneracy, or excluding a fixed graph as a minor, can be defined by hereditary \textsf{CMSO}$_2$ formulae.

The following result encapsulates the technique developed by Gartland, Lokshtanov, Pilipczuk, Pilipczuk, and Rzążewski~\cite{DBLP:conf/stoc/GartlandLMPPR24}.

\begin{theorem}\label{thm:alg-blobclosed}
Let $\cG$ be a hereditary class of graphs such that
\begin{enumerate}
    \item For every $\epsilon >0$ there is a (quasi)polynomial-time algorithm that,
    given an $n$-vertex weighted graph, returns either an independent set in $G$ of weight at least $(1-\epsilon)$ times optimum, or correctly concludes that $G \notin \cG$, and
    \item $G \in \cG$ if and only if $G^\circ \in \cG$.
\end{enumerate}
Let $r \geq 0$, $\varepsilon \in (0,1)$ be a real, and $\psi$ be a hereditary \textsf{CMSO}$_2$ formula.
There is an algorithm that, given a graph $G$, in (quasi)polynomial time returns one of the following outputs:
\begin{itemize}
\item a solution to $(\tw \leq r,\psi)$-\MWIS of size at least $(1-\varepsilon)$ times the optimum one, or,
\item a correct conclusion that no solution to $(\tw \leq r,\psi)$-\MWIS exists, or
\item a correct conclusion that $G \notin \cG$.
\end{itemize}
\end{theorem}

For fixed $k$, let $\cG$ denote the class of graphs with $\treepi_{P_3} \leq k$. 
Combining \cref{thm:alg-blobclosed}, \cref{thm:mwisp4}, and \cref{lem:P3-blob-preservation}, we obtain the following.

\thmapproxcmsopthree*

%% file: sec-dbs.tex

The goal of this section is to prove \cref{thm:dbs-together}, that is, to show that graphs with bounded $\treepi_{\cC}$ or bounded $\treepi_{P_t}$, for any fixed $t$, admit balanced separators dominated by few vertices.
Before we start, we need a few definitions.

Let $(G,\wei)$ be a vertex-weighted graph. We say that a subset $C \subseteq V(G)$ is \emph{light} if $\wei(C) \leq \wei(G)/2$, and otherwise $C$ is \emph{heavy}.
A set $S\subseteq V(G)$ is a \emph{balanced separator} of $G$ if every component of $G-S$ is light.
For a function $f : \N \to \R$, we say that a class $\cG$ of graphs admits \emph{$f$-dominated balanced separators} if for every graph $G\in \cG$ and every weight function there exists a set $X \subseteq V(G)$ 
of size at most $f(|V(G)|)$ such that $N[X]$ is a balanced separator of $(G,\wei)$.

We are mainly interested in the case where $f$ is polylogarithmic.

\subsection{Induced packing of long paths}\label{sec:packing-of-longpaths}
For graphs of bounded $\treepi_{P_t}$, the function $f$ can in fact be taken to be $\Oh(kt)$, i.e., a constant.

\begin{lemma}\label{lem:treepi-Pt-dominated-separator}
For every $t,k$, the class of graphs $G$ satisfying $\treepi_{P_t}(G) \leq k$ has $(k+1)t$-dominated balanced separators.
\end{lemma}

\begin{proof}
Let $(G,\wei)$ be a vertex-weighted graph and let $\cT$ be a tree decomposition of $G$ with $\pi_{P_t}(\cT) \leq k$. Even though we will work with some induced subgraphs of $G$,
we will always define light and heavy subsets with respect to $\wei(G)$.

By \cref{lem:balanced-bag} applied to $(G,\wei)$ and $\cT$,
there is a set $B \subseteq V(G)$ such that $\pi_{P_t}(G,B) \leq k$
and every component of $G-B$ is light.
Let $Y$ be the vertex set of a maximum $B$-rooted induced $P_t$-packing, that is,
a largest family of pairwise disjoint and anticomplete induced $t$-vertex paths,
each intersecting $B$.
As $\pi_{P_t}(G,B) \leq k$, we have $|Y| \leq kt$.

Note that in the graph $G - N[Y]$, there is no $(B \setminus N[Y])$-rooted induced $P_t$,
as its existence would contradict the maximality of $Y$.

If every component of $G-N[Y]$ is light, then we are done: $N[Y]$ is a $kt$-dominated balanced separator.
Thus, assume that $G-N[Y]$ has a heavy component $G'$; it must be unique.
Note that $G'$ must intersect $B$, as every component of $G-B$ is light.
Pick any $v \in B \cap V(G')$ and apply \cref{lem:gyarfas-path} to $(G',\wei)$ and $v$.
Let $X$ be the vertex set of the resulting induced path starting at $v$.
As $G'$ has no $P_t$ rooted at $B \cap V(G')$, we have $|X| \leq t-1$.
Furthermore, every component of $G'-N[X]$ is light, and thus every component of $G-N[X \cup Y]$ is light.

Consequently, $N[X \cup Y]$ is a balanced separator of $(G,\wei)$ dominated by at most $|X| + |Y| \leq (k+1)t$ vertices. This completes the proof. 
\end{proof}

\subsection{Induced packing of cycles}
Now let us turn to graphs of bounded $\treepi_{\cC}$, where $\cC$ is the family of all cycles.
The strategy will be similar to the one used in the previous subsection: (1) apply \cref{lem:balanced-bag} to find a bag $B$ of a tree decomposition that splits the graph into light components, (2) remove a constant number of neighborhoods to eliminate all cycles intersecting $B$, and (3) split the unique heavy component of the resulting graph. However, the arguments will be slightly more complicated.

In particular, for step (2) we cannot proceed similarly as for graphs of bounded $\treepi_{P_t}$,
as a maximum $B$-rooted induced cycle packing can use an unbounded number of vertices even if $\pi_{\cC}(G,B)$ is bounded.
However, we can use the following result of Ahn and Kwon~\cite{Anh-Kwon} who proved that cycles intersecting a given set have the \emph{induced Erdős--Pósa property}.\footnote{The statement of Ahn and Kwon~\cite{Anh-Kwon} concerns $(\ell,B)$-cycles that are defined as non-necessarily induced -- but pairwise anticomplete -- cycles of length at least $\ell$, each intersecting $B$. However, every $(3,B)$-cycle contains an induced cycle intersecting $B$, and every induced cycle intersecting $B$ is a $(3,B)$-cycle. Hence the statement of Ahn and Kwon~\cite{Anh-Kwon} for $\ell=3$ is equivalent to the statement of \cref{thm:AK}.}

\begin{theorem}[Ahn and Kwon~\cite{Anh-Kwon}]
\label{thm:AK}
For all positive integers $k$, every graph $G$, and every
set $B \subseteq V(G)$, either $G$ contains
\begin{itemize}
    \item $k$ pairwise anticomplete induced $B$-rooted cycles, or
    \item a set $Y \subseteq V(G)$ with $|Y|=\Oh(k^3)$ such that  $G-N[Y]$ has no $B$-rooted cycle.
\end{itemize}
\end{theorem}

Now we are ready to prove that graphs of bounded $\treepi_{\cC}$ admit balanced separators dominated by a logarithmic number of vertices.

\begin{lemma}
\label{lem:treepi-Cge3-log-dominated-separator-weighted}
For every $k$, the class of graphs $G$ satisfying $\treepi_{\cC}(G) \leq k$ has $f$-dominated balanced separators for $f(n) = \Oh( k^3 \log n )$.
\end{lemma}

\begin{proof}
The proof is by induction on $|V(G)|$. The base case is trivial, as for $|V(G)|=1$ the unique vertex forms a balanced separator.
Thus, suppose that $n \geq 2$ and for every proper induced subgraph of $G$ with $n'$ vertices and every weight function $\wei'$, $(G',\wei')$ has a balanced separator dominated by $c \cdot (2k+1)^3 (\log n'+1)$ vertices, where $c$ is a constant hidden in the $\Oh(\cdot)$ notation in \cref{thm:AK}.
We aim to prove that $(G,\wei)$ has a balanced separator dominated by $c \cdot (2k+1)^3 (\log n +1)$ vertices.

Let $(G,\wei)$ be a vertex-weighted graph and let $\cT$ be a tree decomposition of $G$ with $\pi_{\cC}(\cT) \leq k$. Again, for induced subgraphs of $G$ and weight function $\wei$,
we define light and heavy subsets with respect to $\wei(G)$.

By \cref{lem:balanced-bag} applied to $(G,\wei)$ and $\cT$,
there is a set $B_{\wei} \subseteq V(G)$ such that $\pi_{\cC}(G,B_{\wei}) \leq k$
and every component of $G-B_{\wei}$ is light.
Similarly, applying \cref{lem:balanced-bag} to $G$, uniform weight function, and $\cT$ again,
we obtain a set $B_\mathsf{1} \subseteq V(G)$ such that $\pi_{\cC}(G,B_\mathsf{1}) \leq k$ and every component of $G-B_\mathsf{1}$ has at most $n/2$ vertices.

Define $B = B_{\wei} \cup B_\mathsf{1}$.
Note that $\pi_{\cC}(G,B) \leq 2k$ and every component of $G-B$ is light \emph{and} has at most $n/2$ vertices.

As $\pi_{\cC}(G,B) \leq 2k$, there is no $B$-rooted induced cycle packing of size $2k+1$.
Therefore, applying \cref{thm:AK} to $G,B$, and $2k+1$ gives the second outcome:
a set $Y$ of size at most $c (2k+1)^3$ such that $G-N[Y]$ has no $B$-rooted cycle.

If every component of $G-N[Y]$ is light, then $N[Y]$ is a balanced separator dominated by $c (2k+1)^3$ vertices and we are done.
Otherwise, let $G'$ be the unique heavy component of $G-N[Y]$ and let $B' = B \cap V(G')$.
Note that $G'$ intersects $B'$ and has no cycle intersecting $B'$.
Let $C_1,\ldots,C_r$ be the components of $G'-B'$. Each of them is light and has at most $n/2$ vertices.

Let $G''$ be obtained by contracting each $C_i$ into a single vertex $c_i$.
We equip $G''$ with the weight function obtained from $\wei$ by setting $\wei(c_i) = \wei(C_i)$ for every $i \in [r]$.
Since $G'$ is connected, $G''$ is connected. Moreover, $G''$ is acyclic: any cycle in $G''$ would expand to a cycle in $G'$ intersecting $B'$, contradicting the choice of $G'$. Hence $G''$ is a tree.

By the centroid property of weighted trees, there exists a vertex $x$ of $G''$ such that every component of $G''-x$ has weight at most $\wei(G'')/2 = \wei(G')/2 \leq \wei(G)/2$, and is therefore light. If $x \in B'$, then every component of $G'-N[x]$ is contained in the expansion of some component of $G''-x$, so $N[Y \cup \{x\}]$ is a balanced separator in $(G,\wei)$ dominated by $c (2k+1)^3 + 1$ vertices.
Thus, suppose that $x = c_i$ for some $i \in [r]$. For brevity, denote $C_i$ by $C$. 

\begin{claim}\label{clm:components-attach-to-single-vertex}
        Every component of $G'-C$ attaches to a single vertex of $C$.
\end{claim}
\begin{claimproof}
        Suppose for a contradiction that there is a component $H$ of $G'-C$ that attaches to two distinct vertices in $C$. Let $u,v$ be two distinct neighbors of $H$ in $C$ so that their distance in $C$ is minimized.
        Let $P$ be a shortest $u$-$v$ path in $C$.
        Then $P$ is induced and $u,v$ are the only vertices of $P$ with neighbors in $H$.
        
        Now, let $u',v'$ be two neighbors of $u$ and $v$ in $H$, respectively, chosen in a way that their distance in $H$ is minimized (possibly $u'=v'$).
        Let $Q$ be a shortest $u'$-$v'$ path in $H$.

        By the choice of $u',v'$, no internal vertex of $Q$ is adjacent to $u$ or $v$. As shortest paths are induced, it follows that $P$, $Q$, and the edges $uu'$ and $vv'$ form an induced cycle in $G'$.
        As $C$ is a component of $G'-B'$, its neighbors in $G'$ are contained in $B'$, and thus the cycle intersects $B'$, a contradiction.
\end{claimproof}

For each $v \in V(C)$, let $A_v$ denote the union of vertex sets of components of $G'-C$ that attach to $v$.
We define a new weight function $\wei'$ on $C$ as follows: $\wei'(v) = \wei(v) + \wei(A_v)$ for every $v \in V(C)$.
By \cref{clm:components-attach-to-single-vertex}, the sets $A_v$ are pairwise disjoint and their union is $V(G') \setminus C$, so $\wei'(C) = \wei(G') \leq \wei(G)$.
Furthermore, $C$ is an induced subgraph of $G$ with at most $n/2$ vertices, so by \cref{lem:structural-package}~\eqref{sprop:3} we have $\treepi_{\cC}(C) \leq k$ and the induction hypothesis applies to $(C,\wei')$.
We obtain a set $X$ such that $N[X]$ is a balanced separator for $(C,\wei')$ and 
\[|X| \leq  c \cdot (2k+1)^3 (\log |V(C)|+1)  \leq c(2k+1)^3 (\log n/2 +1) = c(2k+1)^3 \log n.\]

\begin{claim}
        The set  $N[X \cup Y]$ is a balanced separator for $(G,\wei)$.              
\end{claim}
\begin{claimproof}
        Consider components of $G-N[X \cup Y]$.
        Every such component is either (i) a component of $G-N[Y]$ other than $G'$, or
        (ii) contained in $A_x$ for some $x \in C \cap N[X]$,
        or (iii) the union of a component of $C - N[X]$ and its corresponding sets $A_x$.
        The components of type (i) are light as $G'$ is the unique heavy component of $G-N[Y]$.
        The components of type (ii) are light as they are contained in components of $G' - C$, 
        and every such component is light by the choice of $C$.
        Finally, the weight of a component of type (iii) with respect to $\wei$ is equal to the weight of its corresponding component of $C-N[X]$ with respect to $\wei'$, so it is at most $\wei'(C)/2 = \wei(G')/2 \leq \wei(G)/2$, and hence it is light.        
\end{claimproof}

As the size of $X \cup Y$ is
\[      
   |X \cup Y| =  |Y| + |X| \leq c(2k+1)^3 + c(2k+1)^3 \log n = c(2k+1)^3 (\log n +1),
\]
the proof is complete.
\end{proof}

%% file: sec-computing.tex
In this section we discuss issues related to computing tree decompositions of bounded induced $\cH$-packing number.

\subsection{Approximation algorithm}
First, let us prove \cref{thm:computedecomposition}.
We will use the notion of blob graphs, introduced in \cref{sec:blob}.

Let $\cH \neq \emptyset$ be a fixed finite family of connected graphs.
Let $h=\max\{|V(H)| ~|~ H\in \cH\}$.
For a graph $G$ and a set $B \subseteq V(G)$, we define the following induced subgraphs of $G^\circ$.
\begin{align*}
        G^\circ_\cH  = & \ G^\circ[\{X \subseteq V(G) ~|~ G[X] \text{ is isomorphic to a member of } \cH\}],\\
        G^\circ_\cH(B)= & \ G^\circ_{\cH}[\{X\in V(G^\circ_{\cH}) ~|~ X\cap B\ne\emptyset\}].
\end{align*}

As induced $\cH$-packings in $G$ correspond to independent sets in $G^\circ_\cH$,
we have the following:
\[
        \pi_{\cH}(G,B) = \alpha(G^\circ_{\cH}(B)).
\]

\begin{lemma}
\label{lem:blob-equivalence}
Let $\cH \neq \emptyset$ be a fixed finite family of connected graphs. For every graph $G$,
$\treepi_{\cH}(G) =  \treealpha(G^\circ_{\cH})$.
Furthermore, given a tree decomposition of one graph,
in time polynomial in $|V(G)|$, $|V(G^\circ_\cH)|$, and the size of the decomposition,
one can compute a corresponding decomposition of the other.
\end{lemma}

\begin{proof}
First let $\cT=(T,\beta)$ be a tree decomposition of $G$.
We take the same tree decomposition of $G^\circ_\cH$ as in the proof of \cref{lem:P3-blob-preservation}:
$(T,\beta^\circ)$, where, for every $t \in V(T)$, we define
\[
        \beta^\circ(t)
        =
        \{X\in V(G^\circ_{\cH}) ~|~ X\cap \beta(t)\neq\emptyset\}        
\]
As for every $t \in V(T)$ we have $\alpha(G^\circ_{\cH}[\beta^\circ(t)]) = \alpha(G^\circ_{\cH}(\beta(t))) = \pi_{\cH}(G,\beta(t))$,
we conclude that $\treealpha(G^\circ_{\cH})
        \leq
        \treepi_{\cH}(G)$.

\medskip
Conversely, let $(T,\beta^\circ)$ be a tree decomposition of
$G^\circ_{\cH}$. For $X \in V(G^\circ_{\cH})$, define $T_X = \{ t \in V(T) ~|~ X \in \beta^\circ(t) \}$.
For every vertex $v\in V(G)$, let
\[
        \cX_v := \{X\in V(G^\circ_{\cH}) ~|~ v\in X\}.
\]
Then, we define 
\[
        T_v = \begin{cases}
        \bigcap_{X\in\cX_v}T_X & \text{ if } \cX_v \neq \emptyset,\\
         V(T) & \text{ if } \cX_v = \emptyset.
        \end{cases}
\]

Finally, for every $t \in V(T)$, we define $\beta(t):=\{v\in V(G) ~|~ t\in T_v\}$.

\begin{claim}
        $(T,\beta)$ is a tree decomposition of $G$.
\end{claim}
\begin{claimproof}
First, note that the set of nodes of $(T,\beta)$ whose bags contain $v \in V(G)$ is precisely $T_v$.
Let us argue that it is nonempty and connected.
This is obvious if $\cX_v=\emptyset$.

Otherwise, the sets in $\cX_v$ are pairwise intersecting and thus they form a clique in $G^\circ_\cH$.
Hence, by the properties of tree decompositions, there is a node $t \in V(T)$
such that $\cX_v \subseteq \beta^\circ(t)$.
As the intersection of subtrees is connected, we conclude that nodes in $T_v$ form a nonempty subtree.

What remains to show is that for every edge $xy \in E(G)$,
there is $t \in V(T)$ such that $x,y \in \beta(t)$, i.e., that $T_x\cap T_y \neq \emptyset$.
This is obvious if $\cX_x=\emptyset$ or $\cX_y=\emptyset$,
so assume otherwise.
Take arbitrary $X\in\cX_x$ and $Y\in\cX_y$.
If $X=Y$, then $T_X \cap T_Y = T_X \neq  \emptyset$.
If $X\neq Y$, then $X$ and $Y$ are adjacent in $G^\circ_{\cH}$:
either they intersect, or they are disjoint and there is an edge $xy$ between them.
Hence $T_X \cap T_Y \neq \emptyset$. It follows that the family
$
        \{T_X ~|~ X\in\cX_x\}\cup \{T_Y ~|~ Y\in\cX_y\}
$
is pairwise intersecting. By the Helly property for subtrees of a tree,
its total intersection is nonempty, and this total intersection is
contained in $T_x\cap T_y$. Thus some bag contains both $x$ and $y$.
Hence $(T,\beta)$ is a tree decomposition of $G$.
\end{claimproof}

It remains to show that

\begin{claim}
        For every $t \in V(T)$ it holds that $\pi_{\cH}(G,\beta(t)) \leq \alpha(G^\circ_\cH[\beta^\circ(t)])$.
\end{claim}
\begin{claimproof}
Fix $t\in V(T)$. Let $X$ be a set such that $X \cap \beta(t) \neq \emptyset$ and  $G[X]$ is isomorphic to a member of $\cH$.
Choose $v\in X\cap\beta(t)$.
Since $v\in\beta(t)$, we have $t\in T_v$. Since $X\in\cX_v$, the
definition of $T_v$ implies $t\in T_X$, and hence $X\in\beta^\circ(t)$.
Therefore every induced $\cH$-packing in $G$ whose members
intersect $\beta(t)$ is an independent set of
$G^\circ_{\cH}[\beta^\circ(t)]$. Consequently,
$
        \pi_{\cH}(G,\beta(t))
        \le
        \alpha(G^\circ_{\cH}[\beta^\circ(t)]).
$
\end{claimproof}

Consequently, $\treepi_{\cH}(G) \leq \treealpha(G^\circ_\cH)$,
which, combined with the previous inequality, yields $\treepi_{\cH}(G) = \treealpha(G^\circ_\cH)$.
As both constructions can clearly be performed in polynomial time, the proof is complete.
\end{proof}

\cref{lem:blob-equivalence} allows us to reduce the problem of finding a decomposition with small $\pi_{\cH}$ to the problem of finding a decomposition with small independence number.

We will need the following result.

\begin{lemma}[Dallard, Fomin, Golovach, Korhonen, and Milani\v{c};
Theorem~1 in~\cite{dallard2025computing}]
\label{thm:tree-alpha-approx}
There is an algorithm which, given an $n$-vertex graph $G$ and an integer $k$, runs in
time
$2^{\Oh(k^2)}n^{\Oh(k)}$
and either outputs a tree decomposition of $G$ with independence number at most $8k$,
or correctly reports that $\treealpha(G)>k$.
\end{lemma}

We now derive the corresponding decomposition algorithm for induced
packing treewidth.

\thmcomputedecomposition*

\begin{proof}
Define $h=\max\{|V(H)| ~|~ H\in\cH\}$.
Let \[\cX = \{X\subseteq V(G) ~|~ G[X] \text{ is isomorphic to a member of } \cH\}.\]
Note that $|\cX| \le n^h$ and that $\cX$ can be computed in time $n^{\Oh(h)}$, which is polynomial in $n$ since $h$ is a constant.
If $\cX = \emptyset$, then we can return any tree decomposition $\cT$ of $G$,
for instance the one-node tree with bag $V(G)$. This decomposition satisfies $\pi_{\cH}(G,\mathcal T)=0$.
Thus, assume that $\cX \neq \emptyset$.

We construct the graph $G^\circ_{\cH}$; this can be done in polynomial time since $h$ is a constant.
Note that by \cref{lem:blob-equivalence}, we have $\treepi_{\cH}(G) = \treealpha(G^\circ_{\cH})$.

Run the algorithm from \cref{thm:tree-alpha-approx} on $G^\circ_{\cH}$ with parameter $k$.
If the algorithm reports that $\treealpha(G^\circ_{\cH})>k$, then by \cref{lem:blob-equivalence} we correctly report that $\treepi_{\cH}(G)>k$.
Otherwise, it returns a tree decomposition $\cT^\circ=(T,\beta^\circ)$ of
$G^\circ_{\cH}$ with independence number at most $8k$.
Then, by the algorithmic statement in \cref{lem:blob-equivalence},
we can construct a tree decomposition $\cT=(T,\beta)$ of $G$ such that, for every $t \in V(T)$,
\[
        \pi_{\cH}(G,\beta(t)) \leq \alpha(G^\circ_\cH[\beta^\circ(t)]) \leq 8k.
\]
Hence $\pi_{\cH}(G,\cT) \leq 8k$.
As $h$ is a constant, the running time is as claimed.
\end{proof}

\subsection{Hardness of computing induced packing treewidth}
Now, let us briefly discuss some hardness results related to computing induced packing treewidth and corresponding decompositions. Again, it turns out that many such results can be derived from the corresponding results for tree-independence number.

Let $G,H$ be graphs, where $H$ is connected.
Let $G * H$ be any graph obtained from $G$ by adding, for every $v \in V(G)$,
a copy of $H$ and identifying one of its vertices with $v$.
The vertex set of the copy of $H$ attached to $v$, including $v$ itself, is denoted by $V_v$.

\begin{lemma}\label{lem:addsatelites}
        Let $H$ be a connected graph.
        For any graph $G$ we have $\treepi_{H}(G * H) = \treealpha(G)$.
\end{lemma}
\begin{proof}
        Let $(T,\beta)$ be a tree decomposition of $G$ of independence number $\treealpha(G)$.

        For every $t \in V(T)$, we define $\beta'(t) = \beta(t)$.                
        Furthermore, for each vertex $v \in V(G)$, we pick one node $t_v$ of $T$ such that $v \in \beta(t_v)$,
        add a new node $t'_v$ adjacent to $t_v$, and define $\beta'(t'_v) = V_v$.
        Let $T'$ be the tree obtained this way.
        It is straightforward to verify that $(T',\beta')$ is a tree decomposition of $G * H$.

        Pick any $t \in V(T')$ and consider a largest induced $\beta'(t)$-rooted $H$-packing $\cU$ in $G * H$.
        If $t \in V(T') \setminus V(T)$, then clearly it holds that $\pi_H(G*H,\beta'(t)) = 1$; here we use that $H$ is connected.

        Thus, suppose that $t \in V(T)$, i.e., $\beta'(t) = \beta(t)$.
        Every element of $\cU$ intersects $\beta(t)$ and the intersection vertices form an independent set.
        Thus, we have $\pi_{H}(G*H, \beta'(t)) \leq \alpha(G[\beta(t)]) \leq \treealpha(G)$.
        This shows that $\treepi_{H}(G * H) \leq \treealpha(G)$.

        Now, let us consider a tree decomposition $(T,\beta')$ of $G*H$ witnessing $\treepi_H(G*H)$.
        For each $t \in V(T)$, define $\beta(t) = \beta'(t) \cap V(G)$.
        Then $(T,\beta)$ is a tree decomposition of $G$.

        Pick any $t \in V(T)$, and consider an independent set $I$ contained in $G[\beta(t)]$.
        For every $v \in I$, the set $V_v$ induces a copy of $H$ in $G*H$.
        As $I$ is independent in $G$, these copies are pairwise anticomplete.
        Moreover, every such copy intersects $\beta(t) \subseteq \beta'(t)$.
        Consequently, $|I| \leq \pi_H(G*H,\beta'(t)) \leq \treepi_H(G*H)$.

        Therefore $\alpha(G[\beta(t)]) \leq \treepi_H(G*H)$ for every $t \in V(T)$,
        which yields $\treealpha(G) \leq \treepi_H(G*H)$.
        Combined with the previous inequality, this proves the lemma.
\end{proof}

We know the following hardness results concerning $\treealpha$.

\begin{theorem}[\cite{dallard2024treewidth,dallard2025computing}]\label{thm:treealpha-hardness}~
\begin{enumerate}
        \item It is \NP-hard to decide whether a given graph $G$ satisfies $\treealpha(G) \leq 4$.
        \item It is \Wone-hard to approximate tree-independence number within any constant factor.
        \item Assuming Gap-ETH, there is no $f(k) \cdot n^{o(k)}$-time $g(k)$-approximation algorithm for tree-independence number, for any computable functions $f$ and $g$.
\end{enumerate}
\end{theorem}

As \cref{lem:addsatelites} gives a polynomial-time reduction from computing tree-independence number to computing induced $H$-packing treewidth for every connected graph $H$, and \cref{lem:structural-package}~(\ref{sprop:4}) lets us pass from a singleton family $\{H\}$ to any family $\cH$ containing $H$, we immediately obtain the following.

\thmcomputehardness*

%% file: sec-conclusion.tex
Let us conclude the paper with discussing some possible directions for future research.
\subsection{The case of subdivided claws.}\label{sec:claws}
\input{sec-claws}

\subsection{Other open problems.}
Let us conclude the paper by pointing out some open problems and directions for future research.
First, it would be interesting to see if our quasipolynomial-time algorithms can be improved to polynomial time.

\begin{question}
    For what families $\cH$ can \MWIS and related problems be solved in polynomial time on graphs of bounded induced-$\cH$-packing treewidth?
\end{question}

Next, it would be interesting to see how far the algorithmic tractability of graphs of bounded induced-$\cH$-packing treewidth can be extended to generalizations of \MWIS. In particular, the following question is a natural ``local'' extension of \cref{conj:indminors}.

\begin{question}
    Let $H$ be a planar graph and let $\cH$ be the family of all graphs that contain $H$ as a contraction.
    Is it true that $(\tw \leq r,\psi)$-\MWIS and \lcol{3} are tractable in graphs of bounded induced-$\cH$-packing treewidth?
\end{question}

Finally, it would be interesting to consider structural properties of graphs of bounded induced-$\cH$-packing treewidth for various families $\cH$.
First, recall that for every planar graph $H$,
$H$-induced-minor-free graphs are conjectured to admit balanced separators dominated by few vertices~\cite{gartland2023quasi}.
Could this also be true in the ``local'' setting?

\begin{question}
    Let $H$ be a planar graph and let $\cH$ be the family of all graphs that contain $H$ as a contraction.
    Is it true that graphs of bounded induced-$\cH$-packing treewidth admit balanced separators dominated by a polylogarithmic number of vertices?
\end{question}

Next, it is known that graphs of bounded induced matching treewidth (i.e., $\treepi_{P_2}$) are $\chi$-bounded~\cite{abrishami2025excluding}. How far does this property extend to other families $\cH$?

\begin{question}
    For which families $\cH$ are graphs of bounded induced-$\cH$-packing treewidth $\chi$-bounded?
\end{question}

%% file: sec-claws.tex
Recall that for \MWIS there is also a broader family of graph classes obtained by excluding a fixed forest of subdivided claws as an \emph{induced subgraph}; in this setting, quasipolynomial-time algorithms are known and polynomial-time algorithms are conjectured.

One might wonder if this picture can be extended to the ``local'' setting.
We show that this is not the case, at least if we literally ask for induced subgraphs.

The crucial tool is the so-called \emph{Poljak's subdivision trick}~\cite{Po74}. Let $G$ be any graph and let $G'$ be obtained from dividing any edge of $G$ twice, i.e., an edge $uv$ is replaced by a path $u-x-y-v$, where $x,y$ are new vertices.
Then, $\alpha(G') = \alpha(G)+1$.
Consequently, we can reduce solving \MWIS on $G$ to the same problem on $G'$.

\begin{theorem}
    \label{thm:claw}
    Let $H$ be a subdivided claw.
    The \MWIS problem is \NP-hard in graphs $G$ with $\treepi_{H}(G) = 1$.
    Furthermore, the problem cannot be solved in time subexponential in $V(G)$, unless the ETH fails. 
\end{theorem}
\begin{proof}
    The graph $H$ consists of three paths, each starting at the unique degree-3 vertex and ending at a degree-1 vertex.
    Let $t$ be the number of vertices on the longest of these paths.

    It is known that \MWIS  in cubic graphs is \NP-hard and cannot be solved in subexpoential time, unless the ETH fails~\cite{DBLP:journals/tcs/GareyJS76}.
    Let $G$ be a cubic graph on $n$ vertices.

    Let $G'$ be obtained from $G$ by subdividing each edge $2t$ times. In other words, $G'$ consists of vertices of $G$ called \emph{core vertices} joined by paths, each with $2t$ internal vertices.

    For a core vertex $v$, let $S(v)$ be set of vertices at distance $t$ from $v$. Note that for reach core vertex $v$,
    the subgraph of $G'$ induced by $S(v)$ is isomorphic to the three subdivided claw, each of whose three paths has $t+1$ vertices.
    For each edge $uv$ of $G$, let $\mathsf{mid}(uv)$ be two middle vertices of the path in $G'$ corresponding to the edge $uv$. Note that one of these vertices belongs to $S(u)$ and the other belongs to $S(v)$.

    Since $G'$ was obtained from $G$ by multiple applications of Poljak's subdivision trick, solving \MWIS on $G$ is equivalent to solving \MWIS on $G'$.
    Furthermore, the number of vertices of $G'$ is linear in $n$ since $G$ is cubic and $t$ is a constant.

    Thus, we are left with showing that $\treepi_{H}(G') =1$.
    Construct a tree decomposition $(T,\beta)$ of $G'$ as follows.
    We start with a node $t_{\mathsf{mid}}$ whose bag is 
    \[
        \beta(t_{\mathsf{mid}}) = \bigcup_{uv \in E(G)} \mathsf{mid(uv)}.
    \]
    Then, for every vertex $v$, we create a new node $t_v$ adjacent to $t_{\mathsf{mid}}$ and set its bag to $\beta(t_v) = S(v)$.
    It is straightforward to verify that $(T,\beta)$ is a tree decomposition of $G'$.

    Note that for each $v$, it holds that $\pi_H(G,\beta(t_v)) = 1$. Indeed, $S(v)$ contains a copy of $H$ centered at $v$, and other core vertices at at distance greater than $t$ from $S(v)$ so no vertex from this set belongs to any other copy of $H$.
    Similarly, we have $\pi_H(G,\beta(t_\mathsf{mid})) = 0$, as every core vertex is at distance $t$ from $\beta(t_{\mathsf{mid}})$. Consequently, every vertex in $\beta(t_\mathsf{mid})$ is too far away from core vertices to belong to a copy of $H$.
\end{proof}

The hardness result in \cref{thm:claw} is somehow unsatisfying and it suggests the following refinement of the setting.
Let $H$ be a subdivided claw, and let $\mathsf{subdiv}(H)$ be the family of its all subdivisions.
Observe that a graphs is $H$-free if and only if it is $\mathsf{subdiv}(H)$-free.
However, $\treepi_H(G)$ is not equivalent to  $\treepi_{\mathsf{subdiv}(H)}(G)$, in particular, $\pi_{\mathsf{subdiv{H}}}(G',\beta(t_{\mathsf{mid}}))$ is unbounded (equal to $\alpha(G)$).

Could it be true that \MWIS is tractable in graphs with bounded $\treepi_{\mathsf{subdiv}(H)}$, for any subdivided claw $H$?
The smallest case is for $H$ being the claw, i.e., the three-leaf star $K_{1,3}$.
We think it might be the case and perhaps the right tool for this task are \emph{extended strip decompositions}~\cite{ChudnovskyS08e,gartland2023quasi}.

\begin{question}
    Is \MWIS tractable in graphs of bounded induced-$\mathsf{subdiv}(K_{1,3})$-packing treewidth?
\end{question}